\documentclass{article}

\PassOptionsToPackage{numbers, compress}{natbib}
\usepackage[preprint]{neurips_2024}

\usepackage[utf8]{inputenc} % allow utf-8 input
\usepackage[T1]{fontenc}    % use 8-bit T1 fonts
\usepackage{url}            % simple URL typesetting
\usepackage{amssymb,amsfonts,amsmath,amsthm} %ams
\usepackage{enumerate,enumitem,tikz,graphicx,mathrsfs,eucal,verbatim, bbm, derivative}
\usepackage[hidelinks]{hyperref}
\usepackage{tikz, tikz-cd}
\usepackage{svg}
\usepackage{wasysym}
\usepackage{rotating}
\usepackage{hyperref}       % hyperlinks
\usepackage{xcolor}
\definecolor{abstgreen}{rgb}{0,0.5,0}
\definecolor{hemipurple}{rgb}{0.5,0,0.6}
\usetikzlibrary{arrows.meta, decorations.markings}
\usetikzlibrary{calc} % used in the figure of the hemi-cube
\newcommand{\tor}{\mathbb{T}} % discrete tori
\usepackage{epigraph}
\theoremstyle{plain}% default
\newtheorem{thm}{Theorem}[section]

\newtheorem{exmp}[thm]{Example}
\newtheorem{prop}[thm]{Proposition}
\newtheorem{lemm}[thm]{Lemma}

\theoremstyle{definition}
\newtheorem{defn}{Definition}
\theoremstyle{remark}

\newcommand{\pos}{\mathcal{P}}
\newcommand{\rub}{{\rm R}}
\newcommand{\rubg}{\Gamma \rub}
\newcommand{\rubgm}{\Gamma \rub^{-}}
\newcommand{\perm}{{\rm Sym}}
\newcommand{\alt}{{\rm Alt}}
\newcommand{\simp}{\triangle}
\newcommand{\hoso}{\oslash}
\newcommand{\dito}{\ominus}
\newcommand{\mathleftmoon}{\hoso_k}       % colors

\title{Rubik's Abstract Polytopes}

\author{%
Giovanni Luca Marchetti \\
Department of Mathematics \\
Royal Institute of Technology (KTH) \\
Stockholm, Sweden
}

\begin{document}

\maketitle

\begin{abstract}
We generalize the Rubik's cube, together with its group of configurations, to any abstract regular polytope. We discuss general aspects, and provide a description of the group of the Rubik's simplex in arbitrary dimension, and of Rubik's tori. We sketch an analogous argument for the Rubik's hypercube as well.  
%HEMI% \hemi{Lastly, we provide a complete description of the group of the Rubik's hemi-hypercube, which is non-orientable in odd dimensions, by relating it to a non-rotational version of the Rubik's simplex.}
\end{abstract}

\epigraph{\textit{Who in the world am I? Ah, that's the great puzzle.}}{--Lewis Carroll, Alice's Adventures in Wonderland.}

\vspace{.8cm}
\begin{figure}[th!]
\begin{center}
% $\vcenter{\hbox{\includesvg[width=.95\textwidth]{figs/rubik.svg}}}$
\includegraphics[width=.95\textwidth]{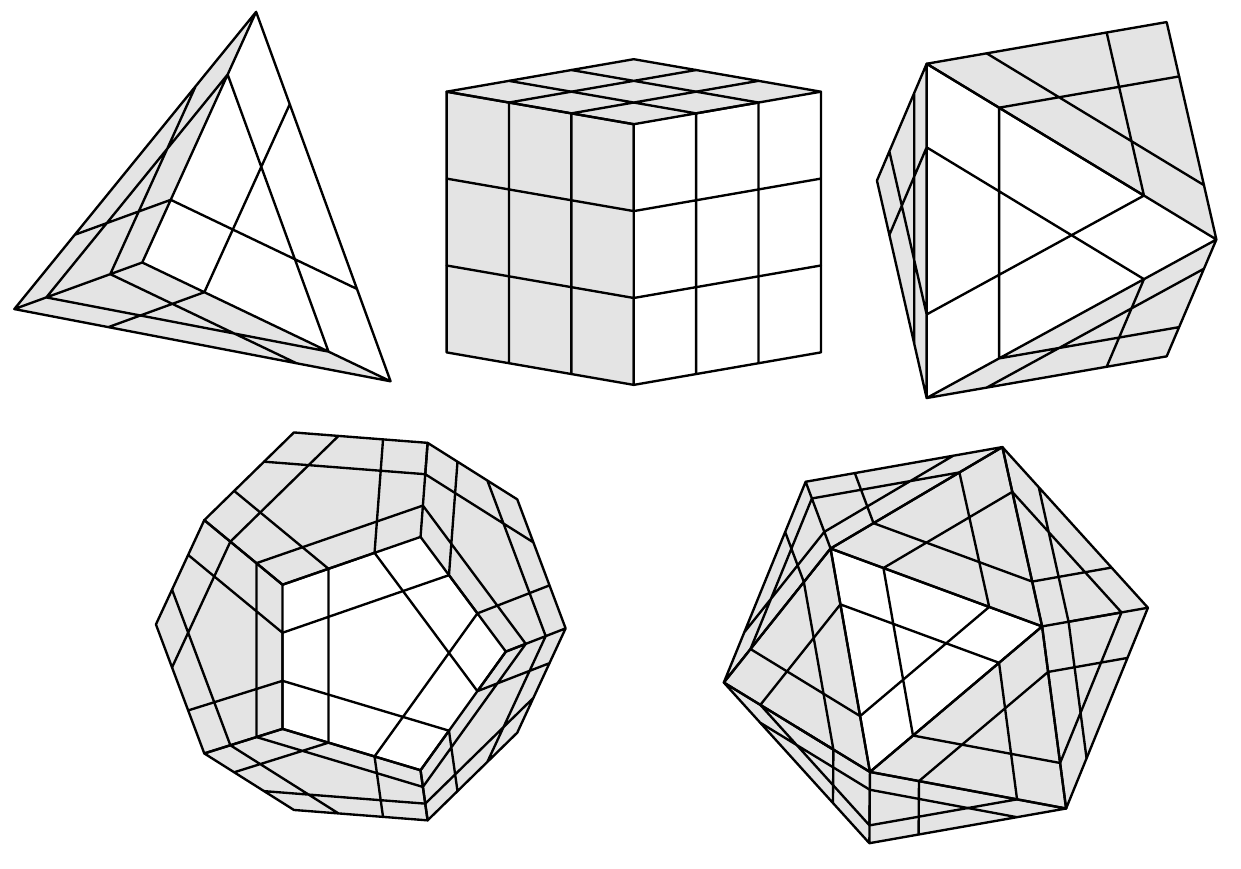}
\vspace{.5cm}
\caption{The Rubik's Platonic solids.}\label{figplato}
\end{center}
\end{figure}

\newpage

\section{Introduction}\label{secintro}
The \emph{Rubik's cube} is a mechanical puzzle invented by Ern\H{o} Rubik in $1974$. It has since become an iconic brain-teaser due to its overwhelming difficulty disguised by an elementary geometric premise. Indeed, the puzzle consists of a cube partitioned into colored pieces that can be permuted by rotating its faces. Starting from an arbitrary configuration, the goal is to reach the initial one characterized by monochromatic faces. There are, however, more than $43$ quintillions of possible configurations that are connected via face rotations. Planning a solution is, therefore, a surprisingly challenging task. 

Throughout the years, the Rubik's cube has inspired a plethora of variations and generalizations -- see \cite{nelson2018abstracting} for an overview. For example, similar physical puzzles have seen light in alternative shapes, such as the Platonic solids (Figure \ref{figplato}). The Rubik's dodecahedron and tetrahedron have been commercialized under the names \emph{Megaminx} and \emph{Halpern-Meier Pyramid} respectively. Even further, higher-dimensional analogues of the Rubik's cube -- the Rubik's \emph{hypercubes} -- have been simulated via software \cite{superlim} and solved (Figure \ref{fighyper}, right). 

From a mathematical perspective, the core interesting aspect of the Rubik's cube and of similar puzzles is that the set of possible configurations forms a group. The latter has face rotations as generators and the initial configuration as the identity element. Solving the puzzle can be rephrased as finding a path to the identity in the Cayley graph associated to the generators. The groups of the Rubik's cube and of its analogues manifest intricate structures that have been described and extensively studied \cite{joyner2008adventures, frey2020handbook}. However, each such group is typically discussed separately, raising the need for a unified mathematical treatment.   
\vspace{.5cm}
\begin{figure}[h!]
\begin{center}
\includegraphics[width=.37\textwidth]{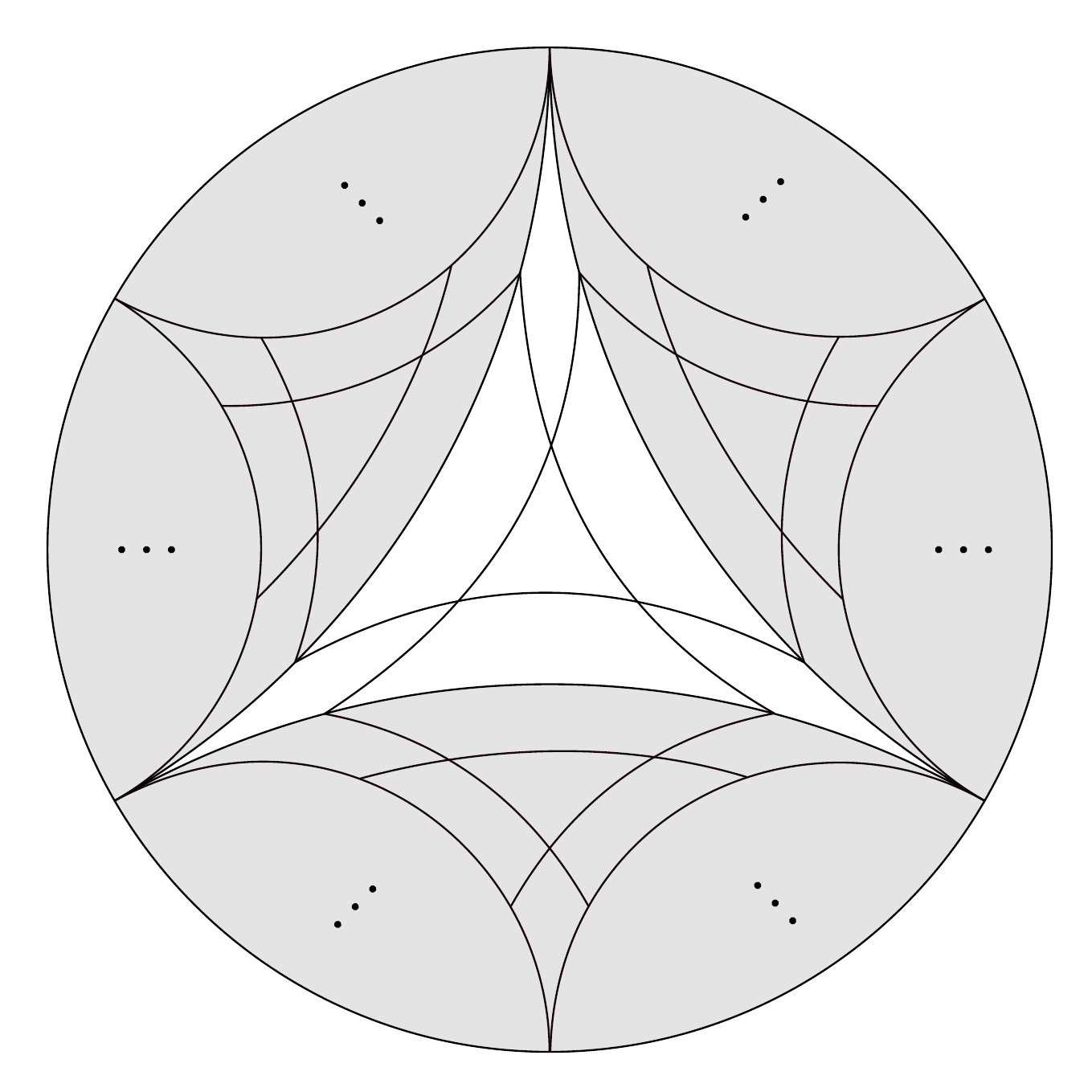} 
\hspace{2cm}
\includegraphics[width=.39\textwidth]{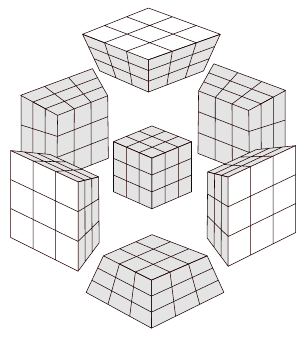}
\vspace{.5cm}
\caption{One the left, the Rubik's hyperbolic triangular tiling. On the right, the Rubik's $4$-dimensional hypercube projected to the $3$-dimensional space. }\label{fighyper}
\end{center}
\end{figure}

In this work, we generalize the Rubik's puzzle, together with the associated group, to arbitrary \emph{regular polytopes}. To this end, we rely on the language of \emph{abstract} polytopes in the sense of \cite{mcmullen2002abstract}. This elegant combinatorial formalism allows us to formulate a general theory that is applicable to polytopes in arbitrary dimensions and in arbitrary geometries. Along the way, we derive constrains on the Rubik's group of a regular polytope generalizing the classical ones for the Rubik's cube, and directly compute the group of Rubik's polygons and Rubik's hosohedra. We then focus on the case of the simplex in arbitrary dimension. In order to describe the associated group, we rely on an inductive argument reducing to the group of the facets, which are lower-dimensional simplices. We then argue that the group of the Rubik's hypercube can be derived similarly, since its facets are lower-dimensional hypercubes as well. The same result for the hypercube has been obtained with a similar inductive argument in \cite{vladislav2020generalization}. Finally, we describe the group of Rubik's tori, exhibiting a non-spherical polytope.    
%HEMI% \hemi{Lastly, we consider the \emph{hemi-hypercubes}, which are obtained from the hypercubes by identifying antipodal faces, and are non-orientable in odd dimensions. We show that their Rubik's groups are governed not by the ones of the hypercubes, but by the non-rotational Rubik's groups of the simplices, which we describe completely as well. This leads to a complete description of the group of the Rubik's hemi-hypercube in arbitrary dimension.}
%Our proof strategy draws inspiration from a video by Mathologer \cite{mathologer} explaining a method to solve the Rubik's tesseract (i.e., a $4$-dimension hypercube) by reduction to the Rubik's cube.

\subsection{Notation and Conventions}
We will make extensive use of permutations and symmetric groups. To this end, we denote by $\perm_k$ the symmetric group consisting of permutations of the finite set $\{ 1, \cdots, k\}$. We follow the functional convention and write the composition product in $\perm_k$ as $(\pi_1 \pi_2)(i) = \pi_1(\pi_2(i))$. Thus, all the expressions involving permutations have to be read from right to left. We denote by ${\rm sign} \colon \perm_k \rightarrow \{ \pm 1\} $ the sign homomorphism and by $\alt_k$ its kernel, i.e., the alternating group consisting of even permutations.

We will repeatedly rely on the following criterion, which is a consequence of a classical theorem of Jordan.

\begin{lemm}\label{jordan}
Let $\Gamma \subseteq \perm_k$ be the subgroup generated by a set of $3$-cycles. Suppose that the graph on $\{1, \cdots, k\}$, where two points are adjacent whenever they belong to the support of a common $3$-cycle of the set, is connected. Then $\Gamma = \alt_k$.
\end{lemm}

\begin{proof}
Since a $3$-cycle acts transitively on its support, the orbits of $\Gamma$ are unions of connected components of the graph, hence $\Gamma$ is transitive. Moreover, $\Gamma$ is primitive. Indeed, consider a system of imprimitivity with blocks of size $b \geq 2$. If a $3$-cycle of the set mapped a block to a different one, then the two blocks would be disjoint subsets of the support of the $3$-cycle, which is impossible since $2b > 3$. Therefore, each $3$-cycle of the set, and thus $\Gamma$, stabilizes every block, which implies by transitivity that there is a single block. By a theorem of Jordan \cite{dixon1996permutation}, a primitive permutation group containing a $3$-cycle contains $\alt_k$. Since $3$-cycles are even, $\Gamma = \alt_k$.
\end{proof}

\section{Abstract Polytopes}\label{intro}
We start by recalling the foundations of the theory of abstract polytopes. Most of the content of this section is distilled from \cite{mcmullen2002abstract}, to which we refer the reader for further details. 

In order to introduce abstract polytopes, we sediment some terminology around partially ordered sets (\emph{posets} for short). Consider a poset $(\pos, \preceq)$. A totally ordered subset of $\pos$ of cardinality $i+1$ is called chain of length $i$. A flag is a maximal chain. Two elements $F,G \in \pos$ are incident if $F \preceq G$ or $G \preceq F$. $\pos$ is connected if for any pair of elements $F,G \in \pos$ there exit a sequence $F=F_0, \cdots, F_n=G \in \pos$ such that $F_i$ and $F_{i+1}$ are incident for every $i$. The interval between $G \preceq F$ is defined as $F / G = \{H \in \pos \ | \ G \preceq H \preceq F \}$. A necessary preliminary definition is the following.
\begin{defn}
An  $n$-dimensional \emph{abstract pre-polytope} is a poset $(\pos, \preceq)$ such that: 
\begin{itemize}
\item $\pos$ has a maximum, denoted by $F_n$, and a minimum, denoted by $F_{-1}$.
\item Each flag in $\pos$ has length $n+1$. 
\item Each interval in $\pos$ is connected.
\end{itemize}
\end{defn}
Elements of a pre-polytope are called faces. It can be shown that intervals of a pre-polytope are pre-polytopes as well. Based on this, the rank of a face is defined as the dimension (i.e., the length of any flag minus one) of $F / F_{-1}$. Faces of rank 
$0, 1, n-2, n-1$ are called vertices, edges, ridges and facets, respectively. Finally, we are ready to introduce the core object of study. 
\begin{defn}\label{diamond}
An $n$-dimensional \emph{abstract polytope} is an $n$-dimensional pre-polytope satisfying the following condition deemed \emph{diamond property}: for each $i < n$, if $G \preceq F$ are faces of ranks $i-1$ and $i+1$ respectively, then there exist exactly two faces $H$ of rank $i$ such that $G \prec H \prec F$.  
\end{defn}

\begin{center}
\begin{figure}[h!]
\[\begin{tikzcd}
	& F \\
	{H_1} && {H_2} \\
	& G
	\arrow[no head, from=3-2, to=2-1]
	\arrow[no head, from=3-2, to=2-3]
	\arrow[no head, from=2-3, to=1-2]
	\arrow[no head, from=1-2, to=2-1]
\end{tikzcd}\]
\caption{Hasse diagram of the interval $F / G$ where $G, H_i$ and $F$ are faces of an abstract polytope of rank $i-1, i$ and $i+1$ respectively. The shape of the diagram motivates the name of the `diamond' property. }
\end{figure}
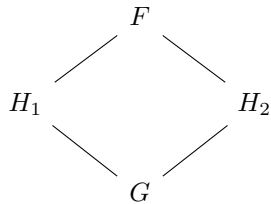
\end{center}
It follows from the diamond property that there exists a unique one-dimensional polytope deemed \emph{segment}. 

An isomorphism between abstract polytopes is a bijective monotone map between them, while an automorphism is an isomorphism between a polytope and itself. The groups of automorphisms is denoted by $\Gamma(\pos)$. By considering the natural action of the latter on subsets of $\pos$, it can be shown that it acts freely on the flags. This motivates the following notion. 
\begin{defn}
An abstract polytope is \emph{regular} if $\Gamma(\pos)$ acts transitively on the set of flags of $\pos$. 
\end{defn}
\begin{figure}
\begin{center}
\includegraphics[width=.85\textwidth]{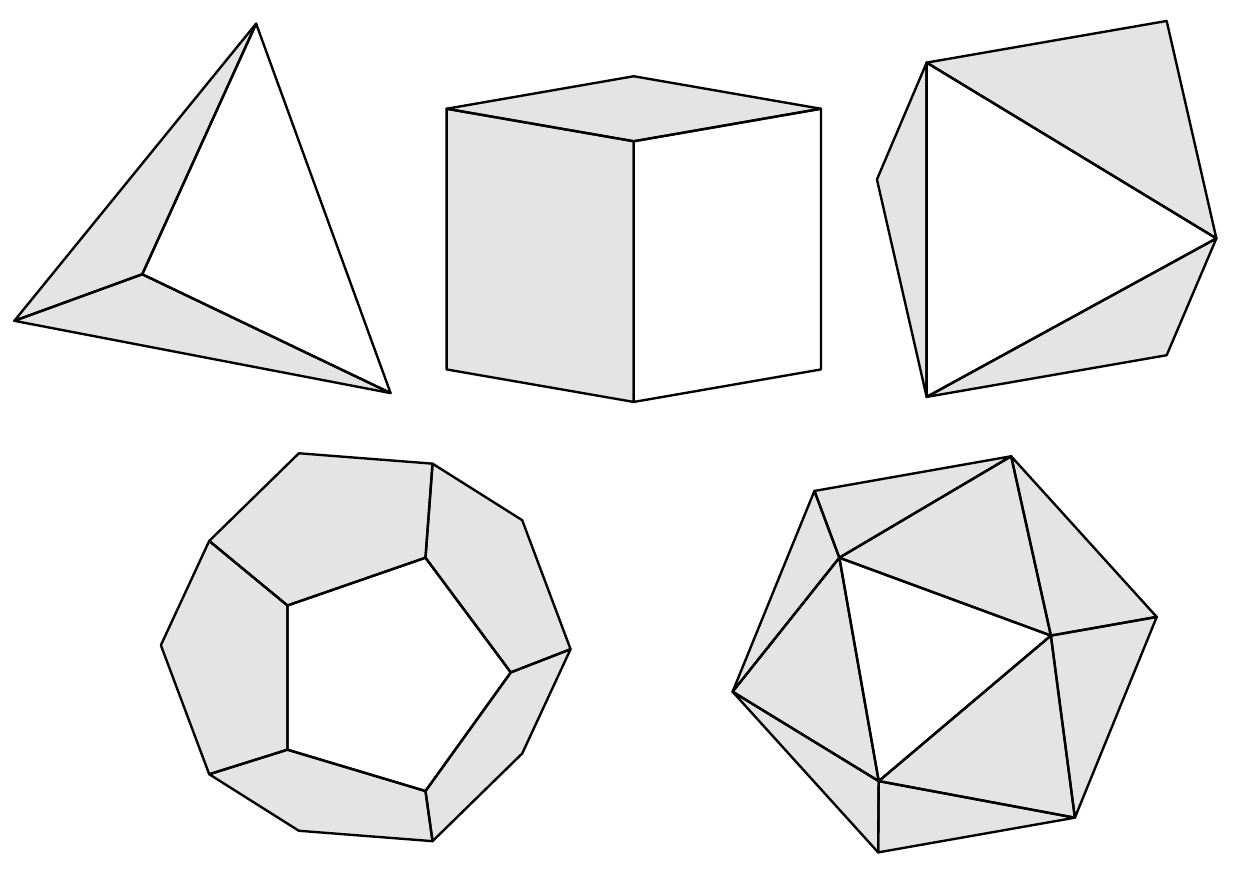}
\vspace{.5cm}
\caption{The Euclidean $3$-dimensional polytopes are known as the Platonic solids. There exist only five of them (up to isomorphism): tetrahedron, cube, octahedron, dodecahedron, and icosahedron.}\label{platonic}
\end{center}
\end{figure}

Consider a regular abstract polytope and fix a flag $F_{-1} \prec F_0 \prec \cdots \prec F_n$. For each $0\leq j < n$, by definition there exists a unique automorphism $\rho_j$ fixing $F_i$ for $i \not = j$ and sending $F_j$ to the only other face $H$ of rank $j$ such that $F_{j-1} \prec H \prec F_{j+1}$ (see Definition \ref{diamond}). It can be shown that $\Gamma(\pos)$ is generated by $\rho_0, \cdots, \rho_{n-1}$ and that the latter satisfy the following relations: 
\begin{itemize}
\item $\rho_j^2 = 1$.
\item $(\rho_j \rho_i)^2 = 1$ for $|j-i| > 1$.
\item $(\rho_{j-1} \rho_j)^{p_j} = 1 $ where $p_j$ is the number of faces of rank $i$ in the interval $F / G$ where $G$ and $F$ are faces of ranks $i-2$ and $i+1$ respectively (which is independent of $F$ and $G$). 
\end{itemize}
The invariants $p_j$ are referred to as Schl\"afli symbols. Note that $\Gamma(\pos)$ satisfies in general more relations than the above ones. In other words, it is a quotient of a Coxeter group of type $A_n$ with Coxeter numbers corresponding to the Schl\"afli symbols. The following is a distinguished subgroup of automorphisms. 
\begin{defn}
The \emph{rotation subgroup} $\Gamma^+(\pos) \subseteq \Gamma(\pos)$ is generated by $\rho_{j-1} \rho_{j}$ for $1 \leq j < n$.
\end{defn}
 $\Gamma^+(\pos)$ is independent of the chosen flag and is a normal subgroup of $\Gamma(\pos)$ of index at most $2$. We say that $\pos$ is \emph{orientable} if this index is exactly $2$. For the present work, the following facts on stabilizers will be necessary. For a face $F$, denote by $\Gamma(\pos, F)$ its stabilizer subgroup (and similarly for rotations).  It can be shown that $\Gamma(\pos, F)$ is isomorphic via restriction to the direct product $\Gamma(F / F_{-1}) \times \Gamma(F_n / F)$. In particular, if $F$ is a facet then $\Gamma(\pos, F)$ is isomorphic to $\Gamma(F / F_{-1})$ (and similarly for rotation subgroups, if $\pos$ is orientable), while if is $F$ a vertex then it is isomorphic to $\Gamma(F_n / F)$. Moreover, if $G$ is a ridge then the isomorphism specializes to $\Gamma(\pos, G) \simeq \Gamma(G / F_{-1}) \times C_2$, where $C_2$ is the groups of order $2$. If $F$ and $F'$ are the two facets incident to $G$ (see the diamond property) then the component in $C_2$ determines whether the given automorphism swaps $F$ with $F'$. In particular, it follows that $\Gamma(\pos, G) \cap \Gamma(\pos, F) = \Gamma(\pos, G) \cap \Gamma(\pos, F')$.
 
%HEMI% \hemi{More generally, the hemi-hypercubes, which are obtained from the hypercubes by identifying antipodal faces, are non-orientable in odd dimensions (see Section \ref{sechemi}).}

 Lastly, we discuss duality for polytopes. The \emph{dual} polytope $\pos^\vee$ of $\pos$ is obtained by reversing the order relation, i.e., $F \preceq G$ in $\pos^\vee$ if and only if $G \preceq F$ in $\pos$. A polytope shares the automorphism group with its dual, and the same holds for the rotation subgroup in the case they are regular. 

\begin{exmp}\label{polygexmp}
{\rm(Polygons)} In dimension $n=2$ the only finite abstract polytopes are the \emph{polygons}, constructed as follows. For an integer $k$, the $k$-gon $P_k$ has $k$ cyclically ordered vertices and an edge between a vertex and the next one. The $k$-gon is regular, its automorphism group is the dihedral group $\Gamma(P_k) \simeq D_{2k}$ of order $2k$ and the rotation subgroup is the cyclic group $\Gamma^+(P_k) \simeq C_k$ of order $k$.     
\end{exmp}

\begin{figure}[h!]
\begin{center}
\begin{tikzpicture}
\node[anchor=south west,inner sep=0] at (0,0) {\includegraphics[width=.75\textwidth]{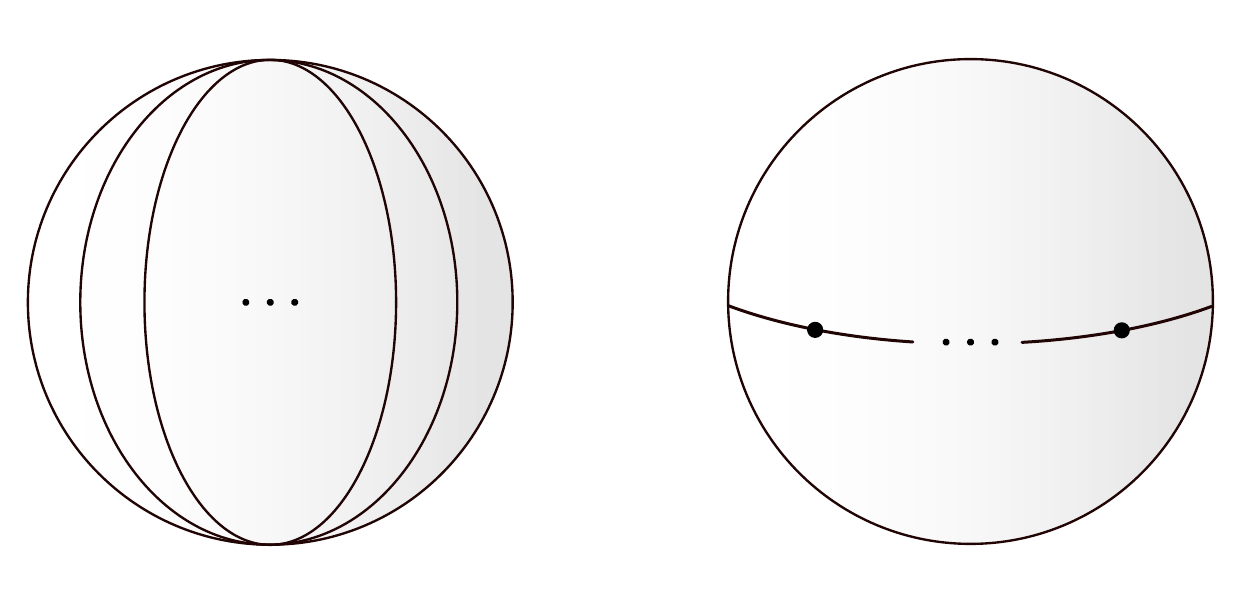}};
\node at (8.1,5) {$\dito(P_k)$};
\node at (2.3,5) {$\hoso(P_k)$};
\end{tikzpicture}
\caption{A hosotope (left) and a ditope (right) of dimension $3$, also referred to as hosohedron and dihedron respectively.}
\end{center}
\end{figure}

\begin{exmp}\label{dihedra}
{\rm(Hosotopes and Ditopes)} A \emph{hosotope} is a polytope with exactly two vertices. Given an $n$-dimensional polytope $\pos$, an $(n+1)$-dimensional hosotope $\hoso(\pos)$ can be constructed by formally adding two faces $\hoso(\pos) = \pos \sqcup \{V_1, V_2 \}$ such that $V_i \prec F$ for each face $F \in \pos \setminus \{F_{-1}, F_n \}$. Any hosotope is obtained via this construction. Moreover, $\hoso(\pos)$ is regular if and only if $\pos$ is. In this case, if $\pos$ is orientable, a remarkable property is that any automorphism of $\pos$ can  be extended to a rotation of $\hoso(\pos)$. Indeed, pick $\phi \in \Gamma(\pos)$. If $\phi \in \Gamma^+(\pos)$, then it is extended by fixing both the $V_i$'s. Otherwise, it is extended by swapping them. It is straightforward to prove that this extension actually results in an isomorphism of groups
\begin{equation}\label{suspiso}
\Gamma(\pos) \simeq \Gamma^+(\hoso(\pos)).
\end{equation}
In particular, $\hoso(\pos)$ is orientable as well, since the automorphism swapping $V_1$ and $V_2$ while fixing all the other faces is not of this form.
Polytopes dual to hosotopes are known as \emph{ditopes}. In other words, a ditope is a polytope with exactly two facets. Again, a ditope $\dito(\pos)$ can be constructed from a polytope $\pos$ by formally adding two facets, and Equation \ref{suspiso} holds for ditopes as well. The precise duality relation reads: 
\begin{equation}
\hoso(\pos)^\vee \simeq \dito(\pos^\vee).
\end{equation}

\end{exmp}

 \section{The Rubik's Construction}
We now introduce the core construction of the present work. Consider an $n$-dimensional abstract regular polytope $(\pos, \preceq)$. 
\begin{defn}\label{rubik}
The \emph{Rubik's construction} $\rub(\pos)$ of $\pos$ is the set of pairs $(F,G)$ where $F,G \not = F_{-1}, F_n$ are faces such that $F \preceq G$.  
\end{defn}
In other words, $\rub(\pos)$ coincides with the ordering $\preceq$ seen as a binary relation on $\pos \setminus \{F_{-1}, F_n \}$. This represents an abstraction of the physical Rubik's cube puzzle. A pair $(F,G) \in \rub(\pos)$ such that $G$ is a facet abstracts a colored piece of the puzzle: $F$ represents the location of the piece (vertex, edge etc.), while $G$ represents the color. Based on this, for an element $(F,G) \in \rub(\pos)$, we refer to the components $F$ and $G$ as its \emph{location} and its \emph{color}, respectively. Differently from the physical puzzle, our definition includes colors of lower rank. This is done purely for the sake of elegance: since an automorphism of a polytope is determined by its action on facets, constraining $G$ to be a facet would result in an equivalent theory. We now describe the moves that can be performed on $\rub(\pos)$. 
 \begin{defn}\label{moves}
Let $H \in \pos$ be a facet and $\phi \in \Gamma^+ (H / F_{-1})$ be a rotation. View $\phi$ as an automorphism of $\pos$ stabilizing $H$ via the isomorphism $\Gamma(\pos, H) \simeq \Gamma(H / F_{-1})$ from Section \ref{intro} (if $\pos$ is orientable, then $\phi$ is a rotation of $\pos$), and define the \emph{move} $\mu_\phi$ associated to $\phi$ as the permutation of $\rub(\pos)$ given by: 
\begin{equation}\label{defmoves}
\mu_\phi(F,G) = \begin{cases}
(\phi(F), \phi(G)) & {\rm if} \ F \preceq H, \\
(F,G) & {\rm otherwise.}
\end{cases}
\end{equation}
The \emph{Rubik's group} $\rubg(\pos)$ is the permutation group on $\rub(\pos)$ generated by the moves $\mu_\phi$ as $H$ varies among the facets of $\pos$ and $\phi$ varies among the rotations of $H$.
\end{defn}

\begin{figure}[h!]
\begin{center}
\vspace{.2cm}
\begin{tikzpicture}
\node[anchor=south west,inner sep=0] at (0,0) {\includegraphics[width=.8\textwidth]{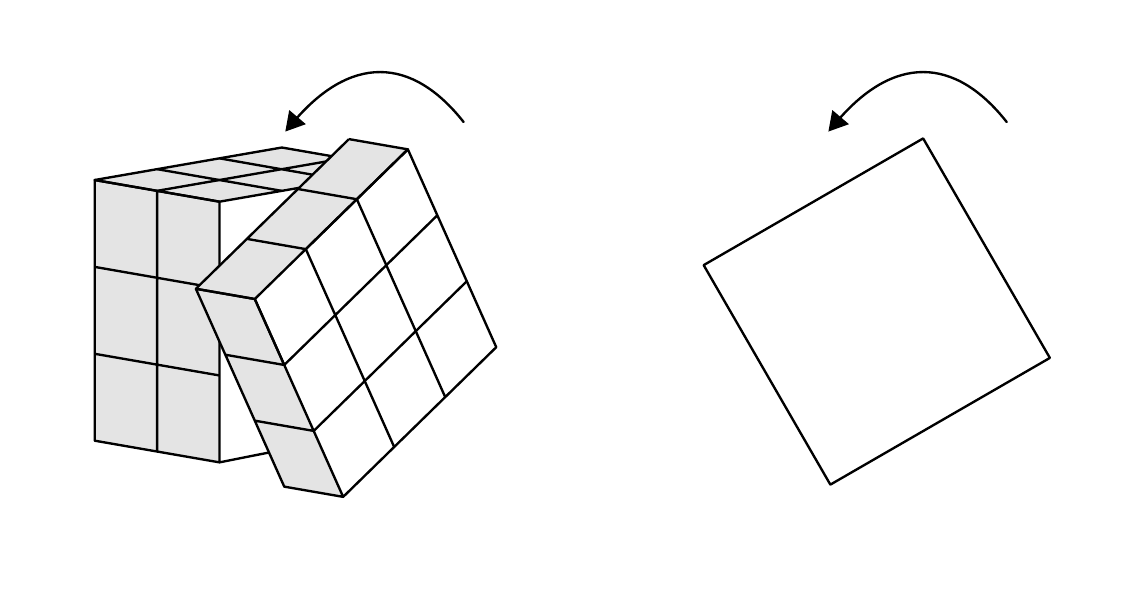}};
\node at (9.4,5.8) {$\phi \in \Gamma^+(F / F_{-1})$};
\node at (3.8,5.8) {$\mu_\phi \in \rubg(\pos)$};
% \draw[red,ultra thick,rounded corners] (7.5,5.3) rectangle (9.4,6.2);
\end{tikzpicture}
\caption{A rotation $\phi$ of a facet $F$ of $\pos$ (right) together with its induced move $\mu_\phi$ on the Rubik's construction $\rub(\pos)$ (left).}
\end{center}
\end{figure}

Just as in the classical physical puzzle, a move is performed by rotating a facet, resulting in the pieces with an incident location to permute accordingly, together with their colors. Note that when $\phi$ is seen as a rotation of $\pos$, the notation for $\mu_\phi$ is slightly ambiguous. Indeed, $\phi$ can in general stabilize multiple facets of $\pos$, and the associated move depends on the choice of the stabilized face. When defining moves, we will make this choice clear from the context.

\begin{exmp}\label{rubdihedra}
{\rm (Rubik's Ditopes)} A trivial example of Rubik's groups is given by ditopes (Example \ref{dihedra}). For a regular polytope $\pos$, consider the associated ditope $\dito(\pos)$. For any of the two facets $H \in \dito(\pos)$, we have that $H / F_{-1} = \pos$. In particular, Equation \ref{defmoves} always falls in its first case, giving an isomorphism of groups
\begin{equation}
\rubg(\dito(\pos)) \simeq \Gamma^+(\pos). 
\end{equation}
In a sense, Rubik's groups generalize rotation groups of regular polytopes. 
\end{exmp}

Rubik's groups can be interpreted as a combinatorial analogue of \emph{holonomy} groups from Riemannian geometry \cite{lee2018introduction}. In this picture, regular polytopes $\pos$ are thought as `discrete manifolds', whose tangent space at $F \in \pos$ is represented by the set $\{G \in \pos \ | \ F \preceq G \}$. The action by the move $\mu_\phi$ mimics the parallel transport of vectors tangent to a facet $H$ by part of a rotation $\phi$ of that facet. Therefore, the Rubik's group is reminiscent of the holonomy group of the tangent bundle of a Riemannian manifold.  

\subsection{Wreath Product Representation}\label{wrrep}
We now provide a general description of $\rubg(\pos)$ together with its first properties. To this end, recall that the wreath product $\Gamma \wr \perm_k$ between an arbitrary group $\Gamma$ and the symmetric group $\perm_k$ on $k$ letters is defined as the semidirect product between the power-group $\Gamma^k$ and $\perm_k$, where the latter acts on the former by permuting components. Our goal is to embed $\rubg(\pos)$ into a direct product of wreath products. Such embedding will not be natural, meaning that its definition requires fixing coordinates on $\pos$ in an appropriate sense. 
\begin{defn}\label{coord}
A \emph{system of coordinates} on $\pos$ consists of regular $i$-dimensional polytopes $\mathcal{A}_i, \mathcal{B}_i$ for $0\leq i<n$ and isomorphisms $ \ \mathcal{A}_i \simeq F_n / F$, $\mathcal{B}_i \simeq F / F_{-1}$ for every face $F$ of rank $i$. Additionally, the following condition must hold: for each pair of vertices $F,G$, there exists a rotation $\phi \in \Gamma^+(\pos)$ sending $F$ to $G$ such that the following diagram commutes: 

\begin{center}
\begin{tikzcd}
	{F_n / F} && {F_n / G} \\
	& {\mathcal{A}_0}
	\arrow["\phi", from=1-1, to=1-3]
	\arrow["\simeq"', from=1-1, to=2-2]
	\arrow["\simeq", from=1-3, to=2-2]
\end{tikzcd}
% \hspace{1cm}
% \begin{tikzcd}
% 	{F / F_{-1}} && {G / F_{-1}} \\
% 	& {B_}
% 	\arrow["\phi", from=1-1, to=1-3]
% 	\arrow["\simeq"', from=1-1, to=2-2]
% 	\arrow["\simeq", from=1-3, to=2-2]
% \end{tikzcd}
\end{center}
\end{defn}

Assume that $\pos$ is finite and fix a system of coordinates. The action by $\rubg(\pos)$ preserves the rank of the locations of elements in $\rub(\pos)$. Thus, an element $\mu \in \rubg(\pos)$ of the Rubik's group induces a permutation $\sigma_i$ of faces  of a given rank $i$. If a face $G$ is sent to some $F$ via such permutation, the restriction of $\mu$ to colors defines an isomorphism $ F_n / G \simeq F_n / F$. The latter can be seen as an automorphism $\tau_i^F \in \Gamma(\mathcal{A}_i)$ via the system of coordinates. Moreover, suppose that $G,F$ are vertices (i.e., $i=0$) and that $\mu=\mu_\psi$ is a move associated to a rotation $\psi \in \Gamma^+(H / F_{-1})$ of a facet $H$ incident to $F$. If $\phi \in \Gamma^+(\pos)$ is the rotation from Definition \ref{coord} then $\phi \psi$ is an automorphism stabilizing $G$ whose restriction to $F_n / G \simeq \mathcal{A}_0$ coincides with $\tau_0^F$. If $\pos$ is orientable, then $\psi$ and $\phi \psi$ are rotations of $\pos$, and it follows from the discussion at the end of Section \ref{intro} that $\tau_0^F$ is a rotation as well. Putting everything together, given an ordering of the faces of $\pos$, the association $\mu \mapsto (\tau_i, \sigma_i)_{0 \leq i < n-1}$ defines an embedding of the Rubik's group into:
\begin{equation}\label{wrprod}
\begin{gathered}
\left(\Gamma^\star(\mathcal{A}_0)\wr \perm_{n_0} \right) \times \prod_{1 \leq i < n-1} \Gamma(\mathcal{A}_i) \wr \perm_{n_i}, \\
\text{where} \quad \Gamma^\star(\mathcal{A}_0) = \begin{cases}
\Gamma^+(\mathcal{A}_0) & \text{if } \pos \text{ is orientable}, \\
\Gamma(\mathcal{A}_0) & \text{otherwise},
\end{cases}
\end{gathered}
\end{equation}
and where $n_i$ denotes the number of faces of rank $i$ in $\pos$. Although isolating vertices makes Definition \ref{coord} and Equation \ref{wrprod} less natural, it provides a tighter embedding.

% \ext{In the non-orientable case, Equation \ref{wrprod} cannot be improved in general. For instance, let $H$ and $K$ be two distinct facets of the hemi-cube, whose vertex figures are triangles. A direct computation shows that $\mu_K \mu_H^2$ fixes the locations of two corner pieces, while exchanging two of the three facets among the colors of each of them. Therefore, the corresponding components of $\tau_0$ are reflections of $\mathcal{A}_0$, and do not belong to $\Gamma^+(\mathcal{A}_0)$. Intuitively, these corner pieces are mirrored after being transported along an orientation-reversing loop in the projective plane.}
%HEMI% \hemi{The Rubik's group of the hemi-cube, and more generally of the hemi-hypercubes, is described completely in Section \ref{sechemi}.}

Recall that the canonical character of a wreath product $\Gamma \wr \perm_k$ is defined as the group homomorphism: 
\begin{equation*}
\begin{aligned}
\Gamma \wr \perm_k & \ \overset{\chi}{\longrightarrow} & \Gamma / \Gamma' \\
(\gamma, \sigma) & \ \longmapsto & \prod_{1 \leq j \leq k} [\gamma^j]
\end{aligned}
\end{equation*}
where $\gamma= (\gamma^1, \cdots, \gamma^k) \in \Gamma^k$, $\Gamma'$ denotes the subgroup of $\Gamma$ generated by commutators $[g,h ] = g^{-1}h^{-1}gh$, and $[\bullet]$ denotes classes in the quotient group $\Gamma / \Gamma'$. Thus, $\Gamma / \Gamma'$ is the Abelianization of $\Gamma$ i.e., its largest Abelian quotient. We now provide a useful formula for computing the canonical character of moves. Both the statement and the proof are closely related to the concept of Artin transfer in group theory (see \cite{isaacs2008finite} for an overview).

\begin{prop}\label{wrprop}
Let $\mu_\phi$ be the move associated to a rotation $\phi \in \Gamma^+(H / F_{-1})$ of a facet $H$, and let $\chi_i$ be the canonical character of the direct factor of Equation \ref{wrprod} corresponding to faces of rank $0 \leq i < n-1$. For each orbit $O$ of the cyclic group generated by $\phi$ acting on the faces of rank $i$ incident to $H$, pick a representative $G_O \in O$. Then $\phi^{|O|}$ stabilizes $G_O$, and:
\begin{equation}\label{monodromy}
\chi_i(\mu_\phi) = \prod_{O} \left[ \omega_{G_O}^{-1} \ \phi^{|O|} \ \omega_{G_O} \right],
\end{equation}
where $\omega_G \colon \mathcal{A}_i \overset{\simeq}{\longrightarrow} F_n / G$ denote the isomorphisms of the system of coordinates, and $\phi^{|O|}$ is restricted to $F_n / G_O$.
\end{prop}
\begin{proof}
The move $\mu_\phi$ fixes the elements of $\rub(\pos)$ whose location is not incident to $H$, hence the corresponding components of $\tau_i$ are trivial. Now fix an orbit $O$, and set $m = |O|$ and $G = G_O$. The move $\mu_\phi$ sends the location $\phi^{j-1}(G)$ to $\phi^{j}(G)$ and acts on colors via $\phi$. Therefore, the component of $\tau_i$ at $\phi^j(G)$ is $\omega_{\phi^j(G)}^{-1} \ \phi \ \omega_{\phi^{j-1}(G)}$ for $j = 1, \cdots, m$. Since $\phi^m(G) = G$, the product of these components, for $j$ decreasing from $m$ to $1$, telescopes to $\omega_G^{-1} \ \phi^m \ \omega_G$. The claim follows since $\chi_i$ takes values in an Abelian group, where the order of the factors is irrelevant.
\end{proof}

Proposition \ref{wrprop} makes the analogy with holonomy discussed above precise: the system of coordinates plays the role of a choice of frames for the `tangent spaces' $F_n / G$, each orbit $O$ is a closed loop of locations along which $\mu_\phi$ transports the pieces, and $\omega_{G_O}^{-1} \phi^{|O|} \omega_{G_O}$ is the holonomy of this loop, i.e., the transformation of the colors of a piece transported once around it, expressed in the frame at $G_O$. Replacing $G_O$ with another element of $O$ conjugates the corresponding factor of Equation \ref{monodromy} by a component of $\tau_i$, so the formula does not depend on the choice of the representatives. Since $\phi \mapsto \mu_\phi$ is a group homomorphism $\Gamma^+(H / F_{-1}) \rightarrow \rubg(\pos)$, the image of $\rubg(\pos)$ lies in the kernel of $\chi_i$ if, and only if, the right-hand side of Equation \ref{monodromy} is trivial for every facet $H$ and every $\phi$ in a set of generators of $\Gamma^+(H / F_{-1})$. This holds, for example, if $\Gamma^+(H / F_{-1})$ admits no non-trivial homomorphism to $\Gamma(\mathcal{A}_i)/\Gamma(\mathcal{A}_i)'$ (to $\Gamma^\star(\mathcal{A}_0)/\Gamma^\star(\mathcal{A}_0)'$ for $i=0$), but not in general (see Section \ref{secfundsimp}).

\subsection{On the Non-Rotational Version}\label{secnonrotgen}
So far, we have defined the Rubik's group via rotational automorphisms, which is in line with the physical Rubik's puzzle. It is of course possible to define a non-rotational version $\rubgm(\pos)$ of the Rubik's group by allowing $\phi$ in Definition \ref{moves} to range among arbitrary automorphisms in $\Gamma(H / F_{-1})$. The wreath product representation from Section \ref{wrrep}
as well as Proposition \ref{wrprop} hold for $\rubgm(\pos)$, albeit without the rotational constraint on vertices in Equation \ref{wrprod}.
%HEMI% \hemi{The non-rotational Rubik's group of the simplex is described in Section \ref{secnonrot}, and is related to the Rubik's group of the hemi-hypercube in Section \ref{sechemi}.}

We start by computing the non-rotational Rubik's group of the simplest regular polytopes, i.e., the polygons $P_k$ from Example \ref{polygexmp}. Note that since facets of polygons are segments with no non-trivial rotations, $\rubg(P_k)$ is trivial instead. The following result describes the group of Rubik's polygons completely. 

\begin{prop}\label{rubpoly}
There is an isomorphism
\begin{equation}
\rubgm(P_k) \simeq \begin{cases}
\perm_k & {\rm if} \ k {\rm \ is \ even} \\
{\rm Ker}(\chi) & {\rm if} \ k {\rm \ is \ odd}
\end{cases}
\end{equation}
where $\chi \colon  C_2 \wr \perm_k \rightarrow C_2$ denotes the canonical character. 
\end{prop}
\begin{proof}
Let $F_1, \cdots, F_k$ and $G_1, \cdots, G_k$ be the vertices and edges of $P_k$ listed in their cyclic order, meaning that $F_i \preceq G_i \succeq F_{i+1}$ for all $i$. We can then define a system of coordinates on $P_k$ by setting $\mathcal{A}_0 = F_n / F_1$ and by identifying $G_{i-1},  G_i \in F_n / F_i$ with $G_k, G_{1} \in \mathcal{A}_0$ respectively. The construction in Section \ref{wrrep} provides then an embedding $\rubgm(P_k) \hookrightarrow C_2 \wr \perm_k$ where $C_2 \simeq \Gamma(\mathcal{A}_0)$ is the cyclic group of order two. Since $G_i / F_{-1}$ is a segment for all $i$, there is a unique non-trivial automorphism $\phi_i \in \Gamma(P_k, G_i) \simeq C_2$ fixing $G_i$. The move $\mu_{\phi_i} \in \rubgm(P_k)$ corresponds to the element $(\tau, \sigma) \in C_2 \wr \perm_k$ where $\sigma = (i \ \ i+1)$ is a transposition and the $j$-th component $\tau^j \in C_2$ of $\tau$ is non-trivial only for $j=i, i+1$. Since transpositions generate the symmetric group, the projection $\rubgm(P_k) \rightarrow \perm_k$ is surjective. We now distinguish the two cases.

\textit{Case 1}: $k$ is even. In this case, the elements $(\tau, \sigma) \in C_2 \wr \perm_k$ lying in the image of $\rubgm(P_k)$ satisfy the following condition: $\tau^j \in C_2$ coincides with $\sigma^{-1}(j) - j$ modulo $2$. Since $k$ is even, this condition defines a subgroup of $C_2 \wr \perm_k$ and it is immediate to verify that it is satisfied by the moves $\mu_{\phi_i}$. This means that $\tau$ is determined by $\sigma$, which in turn implies that the projection $\rubgm(P_k) \rightarrow \perm_k$ is an isomorphism, as desired.  

\textit{Case 2}: $k$ is odd. Given $i$, consider the element of $\rubgm(P_k)$ defined by the following product of $2(k-1)$ moves: 
\begin{equation}\label{polygmovecyc}
  \mu = \mu_{\phi_{i+1}} \cdots \mu_{\phi_{k+i-3}} \mu_{\phi_{k+i-2}}  \mu_{\phi_{k+i-1}} \cdots \mu_{\phi_{i+1}} \mu_{\phi_i}
\end{equation}
where all the indices are modulo $k$. Intuitively, $\mu$ goes back and forth through all the moves in their cyclic order. It is immediate to check that $\mu$ corresponds to $(\tau, \sigma) \in C_2 \wr \perm_k$, where $\sigma = 1$ and $\tau^j$ is non-trivial only for $j = i, i+1$. These elements generate all the $(\tau, 1)$'s in the kernel of $\chi$. Together with the fact that the projection $\rubgm(P_k) \rightarrow \perm_k$ is surjective, this shows that the image of $\rubgm(P_k)$ in $C_2 \wr \perm_k$ is ${\rm Ker}(\chi)$, as desired. 
\end{proof}

We now draw a connection between the rotational and the non-rotational version of the Rubik's group. To this end, we will make use of hosotopes (Example \ref{dihedra}), and in particular of the group isomorphism given by Equation \ref{suspiso}. The following provides a partial description of the Rubik's group of a hosotope $\hoso(\pos)$ in terms of the non-rotational Rubik's group of $\pos$.  

% % introduce the following general construction for polytopes. 
% % \begin{defn}\label{suspdef}
% % Let $\pos$ be an abstract polytope. The \emph{suspension} $\Sigma \pos$ of $\pos$ is defined by adding two faces $S,N$ to the ones of $\pos$ and by extending the order via $S \prec F$ and $N \prec F$ for every face $F \not = F_n, F_{-1}$ of $\pos$. 
% % \end{defn}

% % If $\pos$ is $n$-dimensional then $\Sigma \pos$ is $(n+1)$-dimensional and if $\pos$ is regular then so it $\Sigma \pos$. In this case, the core property of the suspension is that any automorphism of $\pos$ can  be extended to a rotation of $\Sigma \pos$. Indeed, pick $\phi \in \Gamma(\pos)$. If $\phi \in \Gamma^+(\pos)$, then $\phi$ it is extended by fixing both $S$ and $N$. Otherwise, it is extended by swapping them. Is is straightforward to prove that this extension actually results in an isomorphism
% % \begin{equation}\label{suspiso}
% % \Gamma(\pos) \simeq \Gamma^+(\Sigma \pos)
% % \end{equation}

% We wish to discuss the Rubik's group of $\Sigma \pos$. The following gives a partial description in terms of the non-rotational Rubik's group of $\pos$. 
\begin{prop}\label{suspemb}
Suppose that $\pos$ is orientable. There is a group embedding:
\begin{equation}\label{suspembprop}
\rubg(\hoso(\pos)) \hookrightarrow \rubg^-(\pos ) \times \Gamma(\pos),
\end{equation}
with the projection $\rubg(\hoso(\pos)) \rightarrow \rubg^-(\pos )$ being surjective. If $\Gamma(\pos)$ is generated by stabilizers of facets of $\pos$, then the projection $\rubg(\hoso(\pos)) \rightarrow \Gamma(\pos)$ is surjective as well. 
\end{prop}

\begin{proof}
Recall that $\hoso(\pos) = \pos \sqcup \{V_1, V_2 \}$. In particular, for a face $F \in \pos$ the interval $F_n / F$ is the same when $F$ is seen as a face of $\pos$ or of $\hoso(\pos)$. On the other hand, if $F = V_i$ then there is an isomorphism $F_n / F \simeq \pos$. This implies that there exists a decomposition 
\begin{equation}\label{suspdec}
\rub(\hoso(\pos) ) \simeq \rub(\pos) \sqcup \pos   \sqcup \pos.
\end{equation}

Now pick a move $\mu_\phi \in \rubg(\hoso(\pos))$ as in Definition \ref{moves} associated to a rotation $\phi \in \Gamma^+(\hoso(\pos), F)$, where $F$ is a facet of $\hoso(\pos)$ (recall from Example \ref{dihedra} that $\hoso(\pos)$ is orientable, since $\pos$ is). Then $F$ is a facet of $\pos$ as well, $\phi$ restricts on to an automorphism $\hat{\phi} \in \Gamma(\pos, F)$ and $\mu_\phi$ restricts to $\mu_{\hat{\phi}}$ on $\rub(\pos)$. Moreover, $\phi$ fixes both $V_1$ and $V_2$ if and only if $\hat{\phi} \in \Gamma^+(\pos, F)$, in which case $\mu_\phi$ restricts to $\hat{\phi}$ on both the copies of $\pos$ in Equation \ref{suspdec}. If $\hat{\phi} \not \in \Gamma^+(\pos, F)$ then $\phi$ swaps $V_1$ and $V_2$, and induces the isomorphism $\hat{\phi}$ between the two copies of $\pos$. This implies that $\hat{\phi}$ carries all the information of how $\mu_\phi$ acts on $\pos \sqcup \pos$. By associating $\mu_\phi \mapsto (\mu_{\hat{\phi}}, \hat{\phi})$ we obtain a group embedding as in Equation \ref{suspembprop}. Since any automorphism in $\Gamma(\pos, F)$, together with its corresponding move on $\rub(\pos)$, can be obtained by restricting a rotation in $\Gamma^+(\hoso(\pos), F)$ due to Equation \ref{suspiso}, the projection $\rubg(\hoso(\pos)) \rightarrow \rubg^-(\pos )$ is surjective. The last claim is immediate by construction, since $\hat{\phi}$ always fixes a facet of $\pos$.
\end{proof}

% \begin{figure}[h!]
% \begin{center}
% \begin{tikzpicture}
% \node[anchor=south west,inner sep=0] at (0,0) {\includesvg[width=.7\textwidth]{figs/hosohedra.svg}};
% \node at (9,5.2) {$\rub(\mathleftmoon)$};
% \node at (2.5,5.2) {$\mathleftmoon$};
% % \draw[red,ultra thick,rounded corners] (7.5,5.3) rectangle (9.4,6.2);
% \end{tikzpicture}
% \caption{A hosohedron (left) together with its Rubik's construction (right).}
% \end{center}
% \end{figure}

Next, we completely characterize the Rubik's group of hosotopes of dimension $3$, also known as hosohedra. We will denote by $\mathleftmoon = \hoso(P_k)$ the hosohedron associated to the $k$-gon, which has $2$ vertices, $k$ edges and $k$ facets. The following describes the image of the embedding from Proposition \ref{suspemb} for hosohedra.
\begin{prop}
There is an isomorphism 
\begin{equation}
\rubg(\mathleftmoon) \simeq 
\begin{cases}
C_{k/2} \rtimes \perm_k
 & {\rm if } \  k {\rm \ is \ even} \\
C_k \rtimes {\rm Ker}(\chi) & {\rm if } \ k {}\geq 5 {\rm \  is \ odd} \\
{\rm Ker}(\chi) & {\rm if } \ k = 3
\end{cases}
\end{equation}
where $C_k \simeq \Gamma^+(P_k)$ denotes the cyclic group of order $k$, $C_{k/2} \subseteq C_k$ its subgroup of index $2$, and $\chi \colon   C_2 \wr \perm_k \rightarrow C_2$ denotes the canonical character. In the semidirect products, the elements whose permutation in $\perm_k$ is odd act on the cyclic group by inversion.
\end{prop}
\begin{proof}
By Proposition \ref{suspemb}, which applies since polygons are orientable, there is an embedding $\rubg(\mathleftmoon) \hookrightarrow \rubg^-(P_k) \times D_{2k}$, where $D_{2k}=  C_2  \ltimes C_k \simeq \Gamma(P_k)$ denotes the dihedral group, with surjective projection $\rubg(\mathleftmoon) \rightarrow \rubgm(P_k)$, whose kernel we denote by $K$. In what follows, we will adhere to the convention from the proof of Proposition \ref{rubpoly} and denote by $F_1, \cdots, F_k$ and $G_1, \cdots, G_k$ the vertices and edges of $P_k$ in their cyclic order. These correspond to the edges and facets of $\mathleftmoon$ as well. We will additionally denote by $\phi_i \in \Gamma^+(\mathleftmoon, G_i)$ the only non-trivial rotation fixing $G_i$ and by $\hat{\phi}_i \in \Gamma(P_k, G_i)$ its restriction to $P_k$, which is the reflection sending $F_j$ to $F_{2i+1-j}$, where indices are modulo $k$. The moves $\mu_{\phi_i}$ generate $\rubg(\mathleftmoon)$ and correspond to the pairs $(\mu_{\hat{\phi}_i}, \hat{\phi}_i)$. Since $\rho = \hat{\phi}_{i+1}\hat{\phi}_i$ is the rotation sending $F_j$ to $F_{j+2}$ for every $i$, the $\hat{\phi}_i$'s generate a subgroup $E \subseteq D_{2k}$ whose rotations form the cyclic group $E^+ = \langle \rho \rangle$, which is isomorphic to $C_{k/2}$ if $k$ is even and to $C_k$ if $k$ is odd. Moreover, the permutation in $\perm_k$ associated to $\mu_{\hat{\phi}_i}$ is a transposition, while $\hat{\phi}_i$ is a reflection. Therefore, for every element $(\mu, d) \in \rubg(\mathleftmoon)$, $d$ is a rotation if, and only if, the permutation associated to $\mu$ is even. In particular, $K \subseteq E^+$. We now distinguish two cases.

If $k=2$, then $E^+$, and thus $K$, is trivial. If $k=3$, then $K$ is trivial as well, since $D_6$ acts faithfully on the vertices of $P_3$ and $\hat{\phi}_i$ acts on them via the same transposition as $\mu_{\hat{\phi}_i}$. In both cases, $\rubg(\mathleftmoon) \simeq \rubgm(P_k)$, and the claim follows from Proposition \ref{rubpoly}. Assume $k \geq 4$. The moves $\mu_{\hat{\phi}_i}$ and $\mu_{\hat{\phi}_{i+2}}$ commute since they permute disjoint sets of elements of $\rub(P_k)$, and a direct computation in $C_2 \wr \perm_k$ shows that $(\mu_{\hat{\phi}_i}\mu_{\hat{\phi}_{i+1}})^3 = 1$. On the other hand, $(\hat{\phi}_i \hat{\phi}_{i+2})^2 = \rho^{-4}$ and $(\hat{\phi}_i \hat{\phi}_{i+1})^3 = \rho^{-3}$. Therefore, $K$ contains $\rho^{-4}$ and $\rho^{-3}$, and thus coincides with $E^+$. It follows that $\rubg(\mathleftmoon)$ consists of the pairs $(\mu, d) \in \rubgm(P_k) \times E$ such that $d$ is a rotation if, and only if, the permutation associated to $\mu$ is even. The association $\mu \mapsto (\mu, \hat{\phi}_1^{\epsilon})$, where $\epsilon = 0$ if the permutation associated to $\mu$ is even and $\epsilon = 1$ otherwise, is a section of the projection $\rubg(\mathleftmoon) \rightarrow \rubgm(P_k)$. Since conjugation by the reflection $\hat{\phi}_1$ inverts rotations, this provides an isomorphism $\rubg(\mathleftmoon) \simeq E^+ \rtimes \rubgm(P_k)$ as in the statement, and the claim follows from Proposition \ref{rubpoly}.
\end{proof}

\subsection{Extending Moves from Facets}\label{descent}
In this section, we describe a procedure to extend moves of the Rubik's group of a facet of a polytope to moves of the Rubik's group of the whole polytope. Since facets have rank $n-1$, this can be exploited for arguments and constructions that are inductive with respect to the dimension $n$. To this end, consider a facet $F$ of $\pos$ and a facet $G$ of $F / F_{-1}$, which is a ridge (i.e., of rank $n-2$) when seen as a face $\pos$. By the diamond property there is a unique facet $F' \not = F$ of $\pos$ such that $G \prec F' \prec F_n$. Now, a rotation $\phi \in \Gamma^+(G/ F_{-1})$ induces a rotation $\phi' \in \Gamma^+(F' / F_{-1}, G) \simeq \Gamma^+(G / F_{-1})$.  We extend the move $\mu_\phi \in \rubg(F / F_{-1})$ to the move $\mu_{\phi'} \in \rubg(\pos)$. Note that inclusion gives a natural embedding $\rub(F / F_{-1}) \hookrightarrow \rub(\pos)$. With an additional technical assumption, the following result guarantees that the move extension is coherent with the embedding. 

% However, this association does not extend in general to an embedding of the respective Rubik's groups. The following result gives sufficient conditions for this to happen.

\begin{prop}\label{extendprop}
Suppose that, whenever two distinct facets of $\pos$ are incident to a common ridge, every face incident to both of them is incident to that ridge as well, and pick $F, G, F'$ and $\phi$ as above. Then the move $\mu_{\phi'}$, seen as a permutation of $\rub(\pos)$, restricts on $\rub(F / F_{-1})$ to $\mu_\phi$. In other words, the following diagram commutes: 
\begin{center}
\begin{tikzcd}
	{\rub(\pos)} & {\rub(\pos)} \\
	{\rub(F / F_{-1})} & {\rub(F / F_{-1})}
	\arrow[hook, from=2-1, to=1-1]
	\arrow[hook, from=2-2, to=1-2]
	\arrow["{\mu_{\phi'}}", from=1-1, to=1-2]
	\arrow["{\mu_\phi}", from=2-1, to=2-2]
\end{tikzcd}
\end{center}
\end{prop}

\begin{proof}
By hypothesis, every face incident to both $F$ and $F'$ is incident to $G$ as well. It then follows from Definition \ref{moves} that a pair $(H,K) \in \rub(\pos)$ with location $H \preceq F$ is not fixed by $\mu_{\phi '}$ only if $H \preceq G$. Moreover, the extension of $\phi'$ to $\Gamma^+(\pos, F')$ fixes $F$ as well by the discussion at the end of Section \ref{intro} and its restriction to $F / F_{-1}$ coincides with the extension of $\phi$. Therefore, if $H \preceq F$ then $\mu_{\phi'}$ has the same effect as $\mu_\phi$ on $(H, K)$ when the latter is seen as an element of $\rub(\pos)$ and of $\rub(F / F_{-1})$, as desired.   
\end{proof}

The hypothesis of Proposition \ref{extendprop} is satisfied, for example, if $\pos$ is a lattice, and implies that any two facets of $\pos$ have at most one incident ridge. An example of a polytope not satisfying the hypothesis of Proposition \ref{extendprop} is the $2$-gon (also known as digon) i.e., the polygon with two vertices.

We now exploit Proposition \ref{extendprop}, together with Lemma \ref{jordan}, in order to study the permutations of locations induced by the Rubik's group. To this end, assume that $\pos$ is finite, and recall from Section \ref{wrrep} that an element of $\rubg(\pos)$ induces a permutation $\sigma_i$ of the faces of rank $i$ for every $0 \leq i \leq n-2$. We say that such an element is \emph{clean of rank} $r$ if $\sigma_i = 1$ for all $i \neq r$. Moreover, we say that $\rubg(\pos)$ is \emph{alternating at rank} $r$ if the clean elements of rank $r$ induce all the even permutations of the faces of rank $r$. Since a product of clean elements of different ranks induces the corresponding permutation on the faces of each rank, the image of the projection $\rubg(\pos) \rightarrow \prod_{0 \leq i \leq n-2} \perm_{n_i}$ contains $\prod_{0 \leq i \leq n-2} \alt_{n_i}$ if, and only if, $\rubg(\pos)$ is alternating at every rank. The following result provides two sufficient conditions for $\rubg(\pos)$ to be alternating at a given rank, based on commutators of two different kinds.

\begin{prop}\label{alternprop}
Suppose that $n \geq 3$, and that $\pos$ satisfies the hypothesis of Proposition \ref{extendprop}. Fix $0 \leq r \leq n-2$. Then $\rubg(\pos)$ is alternating at rank $r$ if at least one of the following conditions holds:
\begin{itemize}
\item[(M)] There exist two facets $H, K$ and a face $G$ of rank $r$ such that $G$ is the only face of rank between $r$ and $n-2$ incident to both $H$ and $K$. Moreover, $\rubg(\pos)$ is alternating at every rank smaller than $r$.
\item[(D)] It holds that $r \leq n-3$, and there is a facet $F$ such that $F / F_{-1}$ has at least $5$ faces of rank $r$ and $\rubg(F / F_{-1})$ is alternating at rank $r$.
\end{itemize}
Moreover, the first part of condition (M) always holds for $r = n-2$.
\end{prop}

\begin{proof}
We start with some preliminary observations. Since $\pos$ is regular, $\Gamma(\pos)$ acts transitively on the pairs formed by a face of rank $r$ and a facet incident to it. Moreover, for every facet $H$, the rotation group $\Gamma^+(H / F_{-1})$ acts transitively on the faces of rank $r$ of $H / F_{-1}$. Indeed, $\Gamma(H / F_{-1})$ does, and it is generated by $\Gamma^+(H / F_{-1})$ together with any of the automorphisms $\rho_j$ from Section \ref{intro}, defined with respect to a flag of $H / F_{-1}$, since $\Gamma^+(H / F_{-1})$ has index at most $2$ and does not contain $\rho_j$ when the index is $2$. Since $H / F_{-1}$ has dimension $n-1 \geq 2$, we can pick $j \neq r$, so that $\rho_j$ fixes the face of rank $r$ of the flag. Therefore, the orbits of $\Gamma(H / F_{-1})$ and of $\Gamma^+(H / F_{-1})$ on the faces of rank $r$ coincide. As $H / F_{-1}$ has at least two faces of rank $r$ by the diamond property, it follows that for every face $G$ of rank $r$ incident to $H$, and every other such face $G'$, there exists $\phi \in \Gamma^+(H / F_{-1})$ such that the move $\mu_\phi$ sends $G'$ to $G$, and in particular does not fix $G$. Lastly, since $\pos$ is strongly flag-connected \cite{mcmullen2002abstract}, any two facets of $\pos$ can be joined by a sequence of facets in which any two consecutive ones share a ridge.

We now show the following claim: if, for every facet $H$ and any two distinct faces $G, G'$ of rank $r$ incident to $H$, there exists a clean element of rank $r$ inducing a $3$-cycle whose support contains $G$ and $G'$, then $\rubg(\pos)$ is alternating at rank $r$. Indeed, consider the graph from Lemma \ref{jordan} associated to the set of $3$-cycles induced by these elements. In this graph, any two faces of rank $r$ incident to a common facet are adjacent. Since every face of rank $r$ is incident to some facet, and two facets sharing a ridge share the faces of rank $r$ incident to the latter, the graph is connected by the above observations. Lemma \ref{jordan} then implies that the clean elements of rank $r$ induce all the even permutations of the faces of rank $r$, as claimed.

Suppose now that (D) holds, and let $H$ be a facet. Denote by $S_H$ the set of faces of rank $r$ incident to $H$. Since $\pos$ is regular, there is an automorphism of $\pos$ sending $F$ to $H$, which restricts to an isomorphism $F / F_{-1} \simeq H / F_{-1}$. Therefore, $S_H$ has at least $5$ elements, and $\rubg(H / F_{-1})$ is alternating at rank $r$. Pick an even permutation $a$ of $S_H$, and an element $e \in \rubg(H / F_{-1})$ inducing $a$ on $S_H$ and the identity on the faces of $H / F_{-1}$ of any other rank (recall that the facets of $H / F_{-1}$ are fixed by all its moves). Write $e$ as a product of moves of $H / F_{-1}$, and denote by $\hat{e} \in \rubg(\pos)$ the product of the corresponding extended moves. By Proposition \ref{extendprop}, $\hat{e}$ acts on $\rub(H / F_{-1}) \subseteq \rub(\pos)$ as $e$ does. Since every face $G \neq H$ incident to $H$ is the location of an element of $\rub(H / F_{-1})$, e.g., of $(G, G)$, the element $\hat{e}$ permutes the faces incident to $H$ as $e$ does. In particular, it preserves the set of faces incident to $H$, and thus the set of the remaining faces as well. Now pick a rotation $\phi \in \Gamma^+(H / F_{-1})$, and denote by $p$ the permutation of $S_H$ induced by the move $\mu_\phi$, which fixes all the faces not incident to $H$. It follows that the commutator $[\hat{e}, \mu_\phi]$ fixes all the faces not incident to $H$, and permutes the ones incident to $H$ as the commutator of the permutations induced by $e$ and by $\mu_\phi$. Since $e$ induces the identity on the faces of $H$ of rank different from $r$, the commutator $[\hat{e}, \mu_\phi]$ is clean of rank $r$, and induces $[a, p] = a^{-1} p^{-1} a p$ on $S_H$. As $a$ and $\phi$ vary, the permutations $[a, p]$ generate a subgroup of $\alt(S_H)$, which is normal since $a'^{-1} [a, p] \, a' = [a a', p] \, [a', p]^{-1}$. This subgroup is non-trivial, since $p$ can be chosen to be non-trivial by the above observations, and no non-trivial permutation of $S_H$ commutes with all the even ones. Since $S_H$ has at least $5$ elements, $\alt(S_H)$ is simple, and thus coincides with this subgroup. In particular, the clean elements of rank $r$ realize all the $3$-cycles of $S_H$, and the claim implies that $\rubg(\pos)$ is alternating at rank $r$.

Suppose instead that (M) holds, and let $H, K$ and $G$ be as in its statement. We rely on the following elementary fact: if two permutations $\alpha, \beta$ have supports intersecting in a single point $x$, then their commutator $[\alpha, \beta] = \alpha^{-1} \beta^{-1} \alpha \beta$ is the $3$-cycle $(x \ \ \alpha^{-1}(x) \ \ \beta^{-1}(x))$, while if the supports are disjoint, then $\alpha$ and $\beta$ commute. By the above observations, there exist rotations $\phi \in \Gamma^+(H / F_{-1})$ and $\psi \in \Gamma^+(K / F_{-1})$ such that the moves $\mu_\phi$ and $\mu_\psi$ do not fix $G$, and $\phi$ can be chosen so that $\phi^{-1}(G)$ is any given face of rank $r$ incident to $H$ different from $G$. For every $0 \leq i \leq n-2$, the supports of the permutations of faces of rank $i$ induced by $\mu_\phi$ and by $\mu_\psi$ are contained in the sets of faces of rank $i$ incident to $H$ and to $K$, respectively. By (M), these supports intersect exactly in $G$ for $i = r$, and are disjoint for $i > r$. Therefore, the commutator $[\mu_\phi, \mu_\psi]$ induces the $3$-cycle $(G \ \ \phi^{-1}(G) \ \ \psi^{-1}(G))$ on the faces of rank $r$, the identity on the faces of rank larger than $r$, and even permutations on the faces of rank smaller than $r$, since commutators of permutations are even. By the second part of (M), the latter permutations can be undone by multiplying by clean elements of smaller ranks, which results in a clean element of rank $r$ inducing the above $3$-cycle. Since $\Gamma(\pos)$ acts transitively on the pairs formed by a face of rank $r$ and an incident facet, the configuration $(H, K, G)$ can be sent by an automorphism of $\pos$ to a configuration satisfying the first part of (M) for any such pair in place of $(G, H)$. Therefore, the hypothesis of the claim is satisfied, and $\rubg(\pos)$ is alternating at rank $r$.

Lastly, for $r = n-2$, pick a ridge $G$ and the two facets $H, K$ incident to it, given by the diamond property. By hypothesis, every ridge incident to both $H$ and $K$ is incident to $G$, and thus coincides with it. Therefore, $G$ is the only ridge incident to both $H$ and $K$, which is the first part of (M).
\end{proof}

By induction on $r$, Proposition \ref{alternprop} implies that the image of the projection $\rubg(\pos) \rightarrow \prod_{0 \leq i \leq n-2} \perm_{n_i}$ contains $\prod_{0 \leq i \leq n-2} \alt_{n_i}$ as soon as, for every $0 \leq r \leq n-3$, either (D) or the first part of (M) holds. In this case, since $\prod_{0 \leq i \leq n-2} \alt_{n_i}$ is the kernel of the homomorphism $(\sigma_i)_i \mapsto ({\rm sign}(\sigma_i))_i$, the image consists exactly of the tuples $(\sigma_i)_i$ whose signs belong to the subgroup of $\{ \pm 1 \}^{n-1}$ generated by the signs of the permutations induced by the moves. By regularity, the latter can be computed on a single facet $H$, as the signs of the permutations of the faces of $H / F_{-1}$ of each rank induced by the rotations of $H / F_{-1}$. Condition (D) is suited to inductive arguments with respect to the dimension, and will be applied to simplices in Section \ref{simpsec}. Condition (M), instead, will be applied to tori in Section \ref{sectori}, where some pairs of facets share exactly one vertex. 

% Note that neither condition holds for the vertices of the tetrahedron and of the cube, since the Rubik's groups of polygons are trivial, and any two facets sharing a vertex share an edge as well. These polyhedra are therefore treated separately, as base cases of the inductive arguments.

% \subsection{A Homomorphism for Power Polytopes}\label{power}
% In this section, we define a homomorphism between the Rubik's group of a polytope and the Rubik's group of its \emph{power} polytope. The latter is a combinatorial construction which we now recall. To this end, suppose that $\pos$ is an $n$-dimensional polytope with the property that its faces are determined by the set of vertices they contain, and let $v(\pos)$ be the set of vertices of $\pos$. The power $2^\pos$ of $\pos$ is then an $(n+1)$-dimensional polytope whose vertices are binary sequences indexed by the vertices of $\pos$, i.e. $v(2^\pos) = \{ \pm 1 \}^{v(\pos)}$. Given a face $F \in \pos$ of rank $r$ and a vertex $x \in v(2^\pos)$, a face of rank $r+1$ of $2^\pos$ is defined as: 
% \begin{equation}
%     F_x = \left\{ y \in v(2^\pos) \ | \ \forall i \in v(\pos) \setminus F \ y_i = x_i   \right\},
% \end{equation}
% where we identified $F$ with the subset $F \subseteq v(\pos)$ of the vertices it contains. It can be seen that $2^\pos = \{F_x \ | \ F \in \pos, \ x \in v(2^\pos) \} \sqcup \{ \emptyset \}$, with the ordering given by inclusion of subsets of $v(2^\pos)$, is a polytope that is regular if, and only if, $\pos$ is. Moreover, the group of automorphisms $\Gamma(2^\pos)$ can be characterized as follows.  

\section{The Rubik's Simplex}\label{simpsec}
In this section, we study the Rubik's group for a classical regular polytope in arbitrary dimension: the \emph{simplex}. We start by recalling the definition of the latter. 
\begin{defn}
For $n \geq 0$, the $n$-dimensional \emph{simplex} $\simp^n$ is the abstract polytope whose faces are the subsets of $\{0, \cdots, n\}$ ordered by inclusion. 
\end{defn}
For a face $F \in \simp^n$ of cardinality $r+1$, its rank is given by $r$. The simplex has thus $\binom{n + 1}{r + 1}$ faces of rank $r$, where $\binom{\cdot}{\cdot}$ denotes the binomial coefficient. There are natural isomorphisms $F / F_{-1} \simeq \simp^r$ and $F_n / F \simeq \simp^{n-r - 1}$ given by restricting subsets. The group of automorphisms of the simplex can easily be characterized as the symmetric group $\Gamma(\simp^n) \simeq \perm_{n+1}$, naturally acting on $\simp^n$ by permuting subsets. In particular, the simplex is a regular polytope. By choosing the flag $\emptyset \prec \{ 0\}  \prec \{ 0,1 \} \prec \cdots \prec \{0, \cdots ,n \}$, the automorphism $\rho_j$ from Section \ref{intro} corresponds to the transposition $(j \ \ j+1)$ and thus the rotation subgroup is isomorphic to the alternating group $\Gamma^+(\simp^n) \simeq \alt_{n+1}$. For the rest of the section, we assume $n \geq 3$. \\

\begin{figure}[h!]
\begin{center}
\begin{tikzpicture}
\node[anchor=south west,inner sep=0] at (0,0) {\includegraphics[width=.9\textwidth]{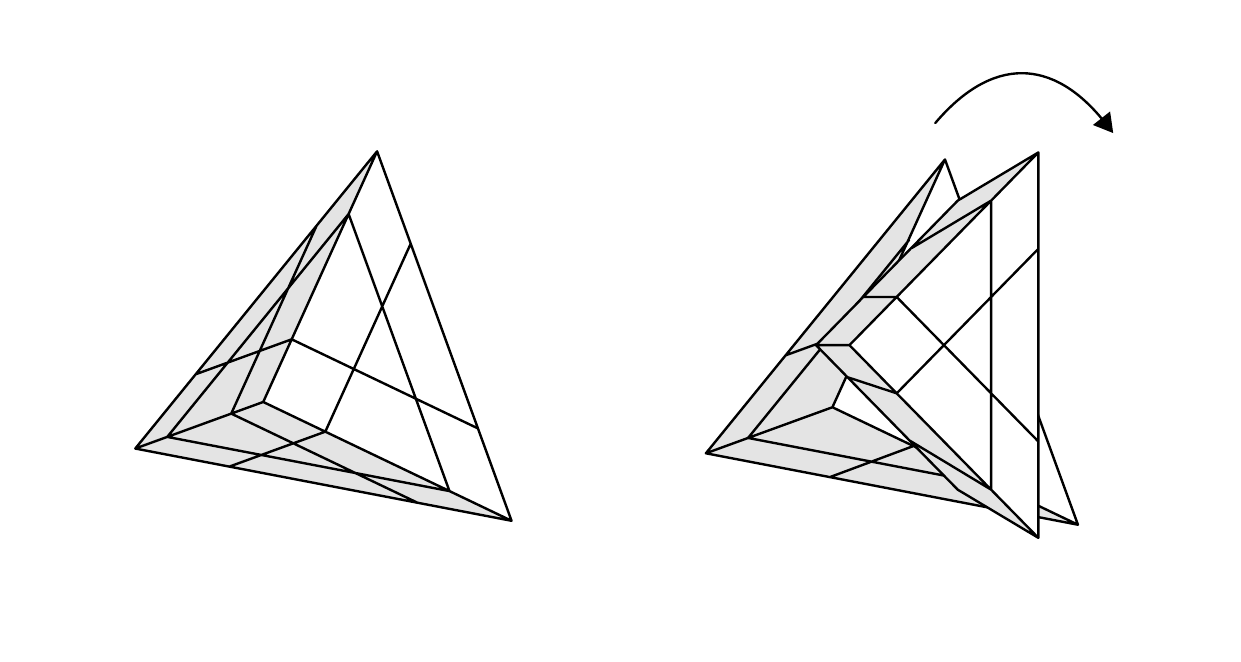}};
\end{tikzpicture}
\vspace{-.5cm}
\caption{A Rubik's tetrahedron $\rub(\simp^3)$ (left) together with a move in $\rubg(\simp^3)$ (right).}
\end{center}
\end{figure}

Fix a system of coordinates on $\simp^n$. By the results of Section \ref{wrrep}, since $\simp^n$ is orientable, $\rubg(\simp^n)$ embeds into
\begin{equation}\label{wrsimp}
\left( \alt_{n} \wr \perm_{n+1} \right) \times \prod_{1 \leq i < n-1} \perm_{n-i} \wr \perm_{n_i}
\end{equation}
where $n_i = \binom{n+1}{i+1}$. We denote a typical element of Equation \ref{wrsimp} by $(\tau_i, \sigma_i)_{0 \leq i \leq n-2}$ where: 
\begin{itemize}
\item $\tau_0 \in \alt_n^{n+1}$.
\item $\tau_i \in \perm_{n-i}^{n_i}$ for $i > 0$.
\item $\sigma_i \in \perm_{n_i}$ for all $i$. 
\end{itemize}
Moreover, we denote the components of each $\tau_i$ by $\tau_i^j \in \perm_{n-i}$, $j=1, \cdots, n_i$. For the Rubik's simplex, a stronger version of the results from Section \ref{descent} holds. Namely, each face $H \in \simp^n$ of rank $r$ induces a group embedding 
\begin{equation}\label{embsimp}
\rubg(\simp^ r) \simeq \rubg(H/F_{-1}) \hookrightarrow \rubg(\simp^n)
\end{equation}
given by extending a permutation $\mu \in \rubg(H / F_{-1})$ to $\bar{\mu} \in \rubg(\simp^n)$ as $\bar{\mu}(F,G) = (\mu( F \cap H) \cup (F  \setminus H), \  \mu( G \cap H) \cup (G  \setminus H)) $. Here, $(\mu(F \cap H), \mu(G \cap H))$ denotes the image of the pair $(F \cap H, G \cap H)$ via $\mu$, with the following conventions for the pairs that do not belong to $\rub(H / F_{-1})$. If $F \cap H = \emptyset$, then $\mu(\emptyset) = \emptyset$, and $\mu(G \cap H)$ is the image of the set $G \cap H$ via the permutation of the vertices of $H$ induced by $\mu$. Otherwise, if $G \cap H = H$, then $\mu(H) = H$, and $\mu(F \cap H)$ is the image of $F \cap H$ via the permutation of the faces of $H$ induced by $\mu$, which fixes $H$ and its facets. For $r=n-1$ this embedding maps moves as described in Section \ref{descent}.

\subsection{Fundamental Result for Rubik's Simplices}\label{secfundsimp}

The main result of this section is the following complete characterization of the group of the Rubik's simplex.

\begin{prop}\label{simpmain}
The Rubik's group of the simplex $\rubg(\simp^n)$ consists exactly of the elements $(\tau_i, \sigma_i)_{0 \leq i \leq n-2}$ satisfying: 
\begin{itemize}
\item $\prod_{1 \leq j \leq 4} \tau_0^j = 1$ for $n = 3$.
\item $\prod_{1 \leq j \leq n_i} \tau_i^j \in \alt_{n-i}$ for $i > 0$.
\item $\sigma_i \in \alt_{n_i}$ for all $i$. 
\end{itemize}
\end{prop}
% Before proving the result, recall that the commutator subgroup $\alt_n'$ is trivial for $n=3$, is the Klein group for $n=4$ and coincides with the whole $\alt_n$ for $n>4$, in which case the first condition is vacuous. \corr{For $n=4$, instead, the first condition fails. Indeed, let $H = \{0,1,2,3\}$ and $\phi = (0 \ 1 \ 2) \in \Gamma^+(\simp^4, H)$. In Equation \ref{monodromy} for $i=0$, the orbit $\{\{0\},\{1\},\{2\}\}$ contributes trivially since $\phi^3 = 1$, while the fixed vertex $\{3\}$ contributes with the restriction of $\phi$ to $F_n / \{3\} \simeq \simp^3$, whose vertices correspond to $\{0,1,2,4\}$. This is a $3$-cycle, which is non-trivial in $\alt_4 / \alt_4' \simeq C_3$, and thus $\chi_0(\mu_\phi) \not = 1$.}

The proof of Proposition \ref{simpmain} will be divided in several steps organized in a series of lemmas. We begin by showing that the conditions are indeed satisfied by the elements in the embedding. 

\begin{lemm}\label{simpfirstlemm}
The elements in $\rubg(\simp^n)$ satisfy the conditions in Proposition \ref{simpmain}.
\end{lemm}
\begin{proof}
The first two conditions follow from Proposition \ref{wrprop}. Fix a facet $H$, so that $\Gamma^+(\simp^n, H) \simeq \alt_n$. For $i > 0$, the canonical character $\chi_i$ takes values in $\perm_{n-i} / \perm_{n-i}' \simeq C_2$, since $\perm_{n-i}' = \alt_{n-i}$. Therefore, $\phi \mapsto \chi_i(\mu_\phi)$ is a homomorphism $\alt_n \rightarrow C_2$, which is trivial since $\alt_n$ has no subgroup of index $2$. This proves the second condition. The first one is vacuous for $n \geq 4$. For $n=3$, a non-trivial $\phi \in \Gamma^+(\simp^3, H) \simeq \alt_3$ cycles the three vertices of $H$, which thus form a single orbit $O$ with $\phi^{|O|} = \phi^3 = 1$. Therefore, $\chi_0(\mu_\phi) = 1$ by Equation \ref{monodromy}.

Regarding the third condition, we will show the following stronger statement: an even permutation in $\alt_k$ for some $k$ induces an even permutation of subsets of $\{1, \cdots, k \}$ of fixed cardinality $h < k$. From here, it is immediate that all the moves (see Definition \ref{moves}) induce even $\sigma_i$'s. In order to deduce the statement, consider a transposition $(i  \  j) \in \perm_k$. Its induced permutation on subsets fixes the ones that contain both $i$ and $j$ and the ones that contain neither of them. Moreover, it induces a transposition on the subsets that contain $i$ but not $j$ by replacing $i$ with $j$. Thus, the induced permutation is a product of $\binom{k-2}{h-1}$ transpositions. This number is independent of $i$ and $j$, meaning that an even permutation in $\alt_k$ induces a product of an even number of permutations that are either all even or all odd. In either case, the induced permutation is even, as desired.
\end{proof}

What is left to show is the converse of the lemma above. This involves factorizing any element as in Proposition \ref{simpmain} into a product of moves in the sense of Definition \ref{moves}. We will construct such a factorization by induction on the dimension $n$. In other words, we will describe a recursive algorithm for `solving' the puzzle represented by the $n$-dimensional Rubik's simplex. Our induction strategy is based on reducing the problems to facets of $\simp^n$, which are $(n-1)$-dimensional simplices, and then relying on the extension of moves from Section \ref{descent} to conclude. One subtlety is that  elements in $\rub(\simp^n)$ whose location is a ridge are unaffected by moves extended from facets, and they have to be addressed separately in the proof. We start by proving the base of the induction.

\begin{lemm}\label{basesimp}
Proposition \ref{simpmain} is true for $n=3$ i.e., for the tetrahedron.
\end{lemm}
\begin{proof}
This is a well-known fact \cite{joyner2008adventures}, and can be even checked by exhaustion in a computer-assisted manner since $|\rubg(\simp^3)| = 3^4 \cdot 4! \cdot 2^6 \cdot 6! \ / \   24 = 3732480$. 
\end{proof} 

As previously anticipated, the next step is to consider elements of $\rub(\simp^ n$) located at ridges, which corresponds to the case of Proposition \ref{simpmain} where $i = n-2$ and, in particular, $n_i = \frac{n(n+1)}{2}$. 

\begin{lemm}\label{sovleridge}
For any two distinct ridges of $\simp^n$, there exists an element of $\rubg(\simp^n)$ inducing trivial permutations of the faces of every rank, whose $\tau_{n-2}$ has exactly two non-trivial components, corresponding to these two ridges.
\end{lemm}

\begin{proof}
Ridges are the complements of the subsets of $\{0, \cdots, n\}$ of cardinality $2$. Given a face $H$ of rank $3$, i.e., a subset of cardinality $4$, we say that a ridge is of \emph{type} $H$ if its complement is contained in $H$. Since the complements of two ridges contain at most $4$ elements in total, any two ridges are of a common type. There are six ridges $F$ of type $H$, and $F \mapsto F \cap H$ identifies them with the edges of $H / F_{-1} \simeq \simp^3$. By the explicit formula for the embedding $\rubg(\simp^3) \simeq \rubg(H / F_{-1}) \hookrightarrow \rubg(\simp^n)$ from Equation \ref{embsimp}, its image acts on the elements of $\rub(\simp^n)$ located at ridges of type $H$ as $\rubg(\simp^3)$ acts on the elements located at edges. Moreover, it fixes the elements located at the other ridges $F$, since $F \cap H$ is either $H$ or a facet of $H / F_{-1}$, which never moves. Lastly, again by the explicit formula, the image of an element of $\rubg(\simp^3)$ inducing trivial permutations of locations induces trivial permutations of the faces of every rank. Therefore, the desired element is the image, for a type $H$ common to the two given ridges, of the element of $\rubg(\simp^3)$ flipping the two corresponding edge pieces and fixing all the other elements of $\rub(\simp^3)$, which exists by Lemma \ref{basesimp}.
\end{proof}

We are now ready for the inductive argument. From here onward, we assume that $n \geq 4$ and that Proposition \ref{simpmain} is true for $n-1$. We will deal with the $\sigma$'s and $\tau$'s from Proposition \ref{simpmain} separately, starting with the former.

\begin{lemm}\label{solveperm}
The projection $\rubg(\simp^n) \rightarrow  \prod_{0 \leq i < n-1}\alt_{n_i}$ is surjective.
\end{lemm} 

\begin{proof}
By Lemma \ref{simpfirstlemm}, the image of the projection is contained in $\prod_{0 \leq i < n-1}\alt_{n_i}$. Therefore, we have to show that $\rubg(\simp^n)$ is alternating at every rank $0 \leq r \leq n-2$, in the sense of Section \ref{descent}. To this end, we apply Proposition \ref{alternprop}, whose hypothesis is satisfied since two distinct facets of $\simp^n$ intersect in a ridge, and every face contained in both of them is contained in this ridge. Let $H$ be a facet and $0 \leq r \leq n-3$. By the inductive hypothesis, Proposition \ref{simpmain} holds for $H / F_{-1} \simeq \simp^{n-1}$, which implies that $\rubg(H / F_{-1})$ is alternating at rank $r$. Moreover, $H / F_{-1}$ has $\binom{n}{r+1}$ faces of rank $r$, which are at least $5$ unless $n = 4$ and $r = 0$. Therefore, by condition (D) in Proposition \ref{alternprop}, $\rubg(\simp^n)$ is alternating at every rank $0 \leq r \leq n-3$, except possibly for $r = 0$ when $n = 4$.

In the latter case, we argue directly. Write $H = \{0, \cdots, 4\} \setminus \{ \alpha \}$, pick an even permutation $a$ of the vertices of $H$, and an element of $\rubg(H / F_{-1})$ inducing $a$ on the vertices of $H$ and the identity on the faces of $H$ of any other rank. By the explicit formula from Equation \ref{embsimp}, its image $\hat{a} \in \rubg(\simp^4)$ induces $a$ on the vertices of $H$ while fixing $\alpha$, the corresponding permutation $\{ v, \alpha \} \mapsto \{ a(v), \alpha \}$ of the edges containing $\alpha$, which is even, and the identity on the faces of any other rank. Since $\rubg(\simp^4)$ is alternating at rank $1$, by multiplying $\hat{a}$ by a clean element of rank $1$ we obtain a clean element of rank $0$ inducing $a$ on the vertices of $H$ and fixing $\alpha$. Since any three vertices of $\simp^4$ are incident to a common facet, it follows that the clean elements of rank $0$ realize all the $3$-cycles of vertices, which generate the alternating group. Therefore, $\rubg(\simp^4)$ is alternating at rank $0$.

Lastly, since $\rubg(\simp^n)$ is alternating at every rank smaller than $n-2$, condition (M) in Proposition \ref{alternprop} holds for ridges, and thus $\rubg(\simp^n)$ is alternating at rank $n-2$ as well.
\end{proof}

In order to proceed, we need the following technical result, which allows us to bring two pieces, together with some of their colors, to a common facet.

\begin{lemm}\label{bringcolors}
Pick two distinct faces $F_1,F_2 \in \simp^n$ of the same rank $ r \leq n -3$ and four facets $G_i^j$, $i,j =1,2$, such that $F_i \preceq G_i^j$ and $G_i^1 \not = G_i^2$ for all $i,j$. Then there exists $\mu \in \rubg(\simp^n)$ such that the locations of the $\mu(F_i, G_i^j)$'s are incident to a common facet of $\simp^n$ different from any of their colors. 
\end{lemm}

\begin{proof}
Denote by $g_i^j$ the only element of $\{0, \cdots, n\}$ not contained in $G_i^j$, and by $W_i$ the complement of $F_i \cup \{ g_i^1, g_i^2 \}$ in $\{0, \cdots, n\}$, which has $n-r-2 \geq 1$ elements. If $w \in W_1 \cap W_2$, then the facet $\{0, \cdots, n\} \setminus \{ w \}$ is incident to both $F_1$ and $F_2$, and is different from all the $G_i^j$'s. Therefore, it is sufficient to find $\mu \in \rubg(\simp^n)$ such that, after replacing the elements $(F_i, G_i^j)$ with their images via $\mu$, the sets $W_1$ and $W_2$ intersect. To this end, consider a facet $H = \{ 0, \cdots, n \} \setminus \{ \alpha \}$ and a rotation $\phi \in \Gamma^+(\simp^n, H) \simeq \alt_{n}$, seen as a permutation of $\{0, \cdots, n\}$ fixing $\alpha$. If $\alpha \in F_i$, then the move $\mu_\phi$ fixes the elements $(F_i, G_i^j)$, while otherwise it replaces $F_i$, $g_i^j$ and $W_i$ with their images via $\phi$. Since $F_1 \not = F_2$ have the same cardinality, there exists $\alpha \in F_1 \setminus F_2$. Suppose first that $W_2 \not = \{ \alpha \}$, and pick $w_2 \in W_2$ different from $\alpha$, and $w_1 \in W_1$, which is different from $\alpha$ as well since $\alpha \in F_1$. Since $\alt_n$ acts transitively on $H$, there exists $\phi$ as above such that $\phi(w_2) = w_1$. Then $\mu_\phi$ fixes the elements $(F_1, G_1^j)$ and replaces $W_2$ with $\phi(W_2)$, which contains $w_1 \in W_1$, as desired. If $W_2 = \{ \alpha \}$ and $F_1 \setminus F_2$ contains an element different from $\alpha$, then we conclude in the same way, with the latter in place of $\alpha$. Otherwise, $F_1 \setminus F_2 = W_2 = \{ \alpha \}$. Since $W_2$ has $n-r-2$ elements, this forces $r = n-3$, and thus $r \geq 1$, since $n \geq 4$. Therefore, there exists $x \in F_1 \cap F_2$, and there are exactly two elements $u, v$ of $\{0, \cdots, n\}$ outside $F_1 \cup F_2$. The move associated to the $3$-cycle $(x \ \ u \ \ v)$, seen as a rotation of $H$, fixes the elements $(F_1, G_1^j)$, and replaces $F_2$ with $(F_2 \setminus \{ x \}) \cup \{ u \}$ and $W_2$ with itself. After this move, $x \in F_1 \setminus F_2$ and $W_2 = \{ \alpha \} \not = \{ x \}$, and we conclude as in the first case, with $x$ in place of $\alpha$.
\end{proof}

\begin{lemm}\label{solveorient}
For any $\tau = (\tau_i)_{0 \leq i \leq n-2}$ with $\tau_0 \in \alt_n^{n+1}$ and $\tau_i \in \perm_{n-i}^{n_i}$ for $i > 0$ satisfying the first two conditions in Proposition \ref{simpmain}, the element $(\tau_i, \sigma_i = 1 )_{0 \leq i \leq n-2}$ belongs to (the image of the embedding of) $\rubg(\simp^n)$.
\end{lemm}

\begin{proof}
For $n = 4$, the claim can be checked by a computer calculation, similarly to Lemma \ref{basesimp}. Indeed, the order of $\rubg(\simp^4)$, computed via the Schreier--Sims algorithm, is $24^5 \cdot 5! \cdot 6^{10} \cdot (10!)^2$. This coincides with the number of elements of Equation \ref{wrsimp} satisfying the conditions in Proposition \ref{simpmain}, which, together with Lemma \ref{simpfirstlemm}, implies that $\rubg(\simp^4)$ consists exactly of these elements. We thus assume $n \geq 5$. Since the conditions in Proposition \ref{simpmain} involve each $\tau_i$ separately, it is enough to consider the elements whose $\tau_i$'s are trivial for all $i$ except one, which we denote by $r$.

We start by refining the construction based on commutators from the proof of Proposition \ref{alternprop}, keeping track of colors. Fix a facet $H = \{0, \cdots, n\} \setminus \{ \alpha \}$, pick $e \in \rubg(H / F_{-1})$, and denote by $\hat{e} \in \rubg(\simp^n)$ its image via the embedding from Equation \ref{embsimp}. By the explicit formula for the latter, $\hat{e}$ preserves the set of elements $(F, G) \in \rub(\simp^n)$ with $\alpha \in F$, and acts on each element $(F, G)$ with $F \subseteq H$ as $e$ acts on $(F, G \cap H)$, while keeping $\alpha$ in $G$ if $\alpha \in G$ (if $G = H$, then $\hat{e}$ moves $F$ as $e$ does and keeps the color $H$). Now pick a rotation $\phi \in \Gamma^+(\simp^n, H) \simeq \alt_n$, which fixes $\alpha$, and denote again by $\phi$ the permutation $(F, G) \mapsto (\phi(F), \phi(G))$ of $\rub(H / F_{-1})$. The move $\mu_\phi$ fixes the elements located at faces containing $\alpha$, and acts on the other ones as $\phi$ acts on $(F, G \cap H)$, again keeping $\alpha$ in $G$ (and the color $H$). Therefore, the commutator $[\hat{e}, \mu_\phi]$ fixes all the elements located at faces containing $\alpha$, and acts on the other ones as the commutator $[e, \phi] = e^{-1} \phi^{-1} e \phi$, computed among permutations of $\rub(H / F_{-1})$, acts on $(F, G \cap H)$, in the above sense. The same holds for products of such commutators.

Denote by $L$ the group generated by the commutators $[e, \phi]$ as $e$ and $\phi$ vary. If $\mu_\psi$ is a move of $H / F_{-1}$, then $\phi^{-1} \mu_\psi \phi = \mu_{\phi^{-1} \psi \phi}$ is a move as well. Therefore, conjugation by $\phi$ preserves $\rubg(H / F_{-1})$, which implies that $L$ is a normal subgroup of $\rubg(H / F_{-1})$, since $e'^{-1} [e, \phi] e' = [e e', \phi] [e', \phi]^{-1}$. We claim that the order of the quotient group $\rubg(H / F_{-1}) / L$ divides $3$. Indeed, in the quotient, the class of an element coincides with the ones of its conjugates by the $\phi$'s, since $\phi^{-1} e \phi = e [e, \phi]$. As $\Gamma^+(\simp^n, H)$ acts transitively on the facets of $H / F_{-1}$, the quotient is then generated by the classes of the moves $\mu_\psi$, where $\psi$ varies in $\Gamma^+(K / F_{-1}) \simeq \alt_{n-1}$ for a fixed facet $K$ of $H / F_{-1}$. Moreover, if $\phi$ stabilizes $K$, then $\phi^{-1} \mu_\psi \phi = \mu_{\phi_K}^{-1} \mu_\psi \mu_{\phi_K}$, where $\phi_K \in \Gamma^+(K / F_{-1})$ denotes the restriction of $\phi$. Since every rotation of $K$ arises as such a restriction, the classes of the moves $\mu_\psi$ commute with each other. Therefore, the quotient is Abelian, and is the image of the homomorphism $\psi \mapsto [\mu_\psi]$ from $\Gamma^+(K / F_{-1}) \simeq \alt_{n-1}$. It is thus a quotient of $\alt_{n-1} / \alt_{n-1}'$, which has order $3$ for $n = 5$ and is trivial for $n \geq 6$. This proves the claim. As a consequence, every element of $\rubg(H / F_{-1})$ of order $2$ belongs to $L$, and is thus realized in $\rubg(\simp^n)$, in the above sense, by a product of commutators of the form $[\hat{e}, \mu_\phi]$. We now apply this construction to specific elements, by distinguishing three cases.

\textit{Case 1}: $1 \leq r \leq n-3$. The subgroup of elements in $\perm_{n-r}^{n_r}$ satisfying the second condition in Proposition \ref{simpmain} is generated by the $\tau_r$'s such that $\tau_r^j$ is a transposition for two indices $j$ and is $1$ otherwise. Therefore, it is enough to show the following. For any two distinct faces $F_1, F_2$ of rank $r$ and four facets $G_i^j$, $i,j = 1,2$, with $F_i \preceq G_i^j$ and $G_i^1 \not = G_i^2$, there exists an element of $\rubg(\simp^n)$ with trivial $\sigma$'s, whose $\tau$'s are trivial except at $F_1$ and $F_2$, where they correspond to the automorphisms of $F_n / F_i$ exchanging $G_i^1$ with $G_i^2$ and fixing the other facets. Up to conjugation by an element of $\rubg(\simp^n)$ as in Lemma \ref{bringcolors}, we can assume that $F_1$ and $F_2$ are incident to a common facet $H = \{0, \cdots, n\} \setminus \{ \alpha \}$ different from all the $G_i^j$'s. In particular, $\alpha \in G_i^j$, and $G_i^1 \cap H$ and $G_i^2 \cap H$ are distinct facets of $H / F_{-1}$ incident to $F_i$. By the inductive hypothesis, there exists $e \in \rubg(H / F_{-1})$ with trivial $\sigma$'s, whose $\tau$'s are trivial except at $F_1$ and $F_2$, where they exchange $G_i^1 \cap H$ with $G_i^2 \cap H$. Indeed, $r \leq (n-1) - 2$, and the product of two transpositions is even. Since $e$ has order $2$, it is realized by a product of commutators as above, which, since $\alpha \in G_i^j$, is the desired element.

\textit{Case 2}: $r = 0$. Since $n \geq 5$, the first condition in Proposition \ref{simpmain} is vacuous. As double transpositions generate $\alt_n$, it is then enough to show that for every vertex $F$ and four distinct facets $G^1, \cdots, G^4$ incident to $F$, there exists an element of $\rubg(\simp^n)$ with trivial $\sigma$'s, whose $\tau$'s are trivial except at $F$, where it corresponds to the automorphism of $F_n / F$ exchanging $G^1$ with $G^2$ and $G^3$ with $G^4$, and fixing the other facets. Since $F$ is incident to $n \geq 5$ facets, there exists a facet $H = \{0, \cdots, n\} \setminus \{ \alpha \}$ incident to $F$ and different from all the $G^k$'s. By the inductive hypothesis, there exists $e \in \rubg(H / F_{-1})$ with trivial $\sigma$'s, whose $\tau$'s are trivial except at $F$, where it exchanges $G^1 \cap H$ with $G^2 \cap H$ and $G^3 \cap H$ with $G^4 \cap H$. Indeed, the first condition in Proposition \ref{simpmain} is vacuous in dimension $n-1 \geq 4$ as well. As before, $e$ has order $2$, and the corresponding product of commutators is the desired element.

\textit{Case 3}: $r = n-2$, i.e., ridges. The subgroup of elements in $\perm_2^{n(n+1)/2}$ satisfying the second condition in Proposition \ref{simpmain} is generated by the ones flipping exactly two ridges, i.e., with exactly two non-trivial components. By Lemma \ref{sovleridge}, for any two ridges, there exists an element of $\rubg(\simp^n)$ with trivial $\sigma$'s flipping exactly these two ridges. By Lemma \ref{simpfirstlemm}, its $\tau_i$'s for $i < n-2$ satisfy the first two conditions in Proposition \ref{simpmain}. By the previous cases, they can thus be made trivial by multiplying by an element of $\rubg(\simp^n)$ with trivial $\sigma$'s and trivial $\tau_{n-2}$. This concludes the proof.
\end{proof}

By putting together Lemmas \ref{simpfirstlemm} - \ref{solveorient}, the desired Proposition \ref{simpmain} follows immediately. As a consequence, the number of possible configurations of the Rubik's simplex is: 

\begin{equation}
|\rubg(\simp^n)| =  \frac{1}{c_n 2^{3n-2}} \prod_{0 \leq i < n-1}(n-i)!^{\binom{n+1}{i+1}}\binom{n+1}{i+1}!,
\end{equation}
where:
%$c_n = |\alt_n / \alt_n'|$ is equal to $3$ for $n=3,4$ and to $1$ for $n \geq 5$.
\begin{equation}
c_n = \begin{cases}
3 & n = 3, \\
1 & n \geq 4.
\end{cases}
\end{equation}

\section{The Rubik's Hypercube}\label{seccube}
We finally focus on the higher-dimensional analogue of the classical Rubik's cube. First, we introduce the abstract hypercube. 

\begin{defn}
For $n\geq 0$, the $n$-dimensional \emph{hypercube} is the abstract polytope whose faces are given by 
\begin{equation}
\square^n = \{-1, 0, 1 \}^n \sqcup \{ F_{-1}\},
\end{equation}
where $F_{-1}$ is a formal element, e.g., $F_{-1} = \emptyset$. We denote the $i$-th component of $F \in \square^n \setminus \{ F_{-1} \} $ as $F(i)$. The ordering on the hypercube is then defined as follows: $G \preceq F$ if for every $i$, $F(i) \not = 0$ implies $G(i) = F(i)$.  
\end{defn}
For a face $F \in \square^n \setminus \{ F_{-1} \}$, define $V(F) = \{ i \ | \ F(i) = 0\}$. The rank of $F$ is then given by the cardinality of $V(F)$. There are thus $2^{n-r}\binom{n}{r}$ faces of rank $r$. If $F$ has rank $r$ then there are natural isomorphisms $F / F_{-1} \simeq \square^r$. Associating to $G \succeq F$ the set $V(G) \setminus V(F)$ defines moreover an isomorphism $F_n / F \simeq \simp^{n-r-1}$. The group of automorphisms of the hypercube can easily be characterized as the wreath product $\Gamma(\square^n) \simeq C_2 \wr \perm_n$, where $C_2 = \{ \pm 1\}$ is the cyclic group of order two. The action by an element $(\lambda, \pi) \in C_2 \wr \perm_n$ on a face $F$ can be described as follows: $\pi$ permutes the components of $F$, while $\lambda_i \in C_2$ determines whether the sign of $F(i)$ gets flipped after the permutation. Now fix the flag $F_{-1} \prec \cdots \prec F_n$ where $F_i = (1, \cdots, 1, 0, \cdots, 0)$. Then the automorphism $\rho_j$ from Section \ref{intro} corresponds to $(\lambda, \pi)$, where $\lambda = 1$ and $\pi = (n-j \ \  n-j+1)$ for $j\not = 0$, while $\pi=1$ and $\lambda = (1, \cdots, 1, -1 )$ for $\rho_0$. It follows that the hypercube is a regular polytope and that the rotation subgroup $\Gamma^+(\square^n)$ consists of automorphisms satisfying: 
\begin{equation}\label{cubeautom}
\prod_{1 \leq j \leq n} \lambda_j = {\rm sign}(\pi).
\end{equation} 
Fix a system of coordinates on $\square^n$. By the results of Section \ref{wrrep}, since $\square^n$ is orientable, $\rubg(\square^n)$ embeds into
\begin{equation}\label{wrcube}
\left( \alt_{n} \wr \perm_{2^n} \right) \times \prod_{1 \leq i < n-1} \perm_{n-i} \wr \perm_{n_i}
\end{equation}
where $n_i = 2^{n-i} \binom{n}{i}$. Similarly to Section \ref{simpsec}, we denote a typical element of Equation \ref{wrcube} by $(\tau_i, \sigma_i)_{0 \leq i \leq n-2}$.  Before discussing the complete description of the group of the Rubik's hypercube, we introduce a comparison homomorphism with the Rubik's simplex.

\subsection{From Rubik's Simplices to Hypercubes}\label{seccomparison}
The group of the $(n-1)$-dimensional Rubik's simplex can be embedded into the one of the $n$-dimensional Rubik's hypercube via a peculiar homomorphism. In order to construct the latter, pick $\alpha \in \{1, \cdots, n \}$ and consider the facet $H = \{1, \cdots, n\} \setminus \{ \alpha \}  \in \simp^{n-1}$, as well as the facets $H_+, H_- \in \square^n$ with all vanishing entries except for the one at index $\alpha$, which is set to $1$ and $-1$ respectively. Note that a rotation $\phi \in \Gamma^+(\simp^{n-1}, H) \simeq \alt_{n-1}$ induces a rotation in $\Gamma^+(\square^n, H_+) = \Gamma^+(\square^n, H_{-})$, which in turn defines two moves $\mu_+, \mu_- \in \rubg(\square^n)$. These moves commute since no face of $\square^n$ (different from $F_{-1}$ and $F_n)$ is adjacent to both $H_+$ and $H_-$. As $\alpha$ and $\phi$ vary, it is straightforward to check that the association $\mu_{\phi} \mapsto \mu_{+}\mu_{-} = \mu_{-}\mu_{+}$ extends to a group embedding
\begin{equation}
\Theta \colon \rubg(\simp^{n-1}) \hookrightarrow \rubg(\square^n).
\end{equation}
Explicitly, for $\mu = (\tau_i, \sigma_i )_{0 \leq i \leq n-3} \in \rubg(\simp^{n-1}) $ (notation according to Section \ref{simpsec}), $\Theta(\mu)$ sends $(F, G) \in \rub(\square^n)$ to $(F', G')$, defined as follows. Assuming the identification $F_n / F \simeq \simp^{n-r-1}$, where $r$ is the rank of $F$, and $r \geq 1$, we have $G' = \tau_{r-1}^j(G)$, where $j$ is the index corresponding to $V(F) \in \simp^{n-1}$. Moreover, $V(F') = \sigma_{r-1}(V(F))$ while the non-vanishing entries of $F'$ are obtained by permuting the ones of $F$ via $\tau_{r-1}^j$. If $r = 0$, instead, the entries of both $F$ and $G$ are permuted via $\sigma_0$.

It is worth mentioning that the embedding described in this section can be generalized to \emph{power polytopes} -- see \cite{mcmullen2002abstract}, Chapter $8$. The power of a polytope is a construction generalizing the relation between the hypercube and the simplex. The definition of $\Theta$ extends almost verbatim to an embedding of the Rubik's group of a polytope into the Rubik's group of its power. However, for the sake of conciseness, we will not discuss further details here. 

\subsection{Fundamental Result for Rubik's Hypercubes}
We now provide a complete characterization of the Rubik's group of the hypercube. As mentioned in Section \ref{secintro}, the same result has been obtained in \cite{vladislav2020generalization} with a similar inductive argument. 

\begin{prop}\label{cubemain}
The Rubik's group of the hypercube $\rubg(\square^n)$ consists exactly of the elements $(\tau_i, \sigma_i)_{0 \leq i \leq n-2}$ satisfying: 
\begin{itemize}
\item $\prod_{1 \leq j \leq 2^n} \tau_0^j \in \alt_n'$.
\item $\prod_{1 \leq j \leq n_i} \tau_i^j \in \alt_{n-i}$ for $i > 0$.
\item $\sigma_i \in \alt_{n_i}$ for $i < n-3 $.
\item ${\rm sign}(\sigma_{n-2}) = {\rm sign}(\sigma_{n-3})$. 
\end{itemize}
\end{prop}
Recall that the commutator subgroup $\alt_n'$ is trivial for $n=3$, is the Klein group for $n=4$ and coincides with the whole $\alt_n$ for $n>4$, in which case the first condition is vacuous. We begin by showing the following. 

\begin{lemm}\label{cubefirstlemm}
The elements in $\rubg(\square^n)$ satisfy the conditions in Proposition \ref{cubemain}.
\end{lemm}

\begin{proof}
% We will prove the result only partially. Namely, we will show that the elements in $\rubg(\square^n)$ satisfy the conditions. The converse is omitted since it can be seen via an argument extremely similar to the one for the simplex, with the details being more tedious. Indeed one can proceed analogously to Proposition \ref{simpmain} by induction on the dimension $n$. This in turn involves showing the analogue for $\rub(\square^n)$ of the reduction result from Lemma \ref{facetlemm}, dealing with ridges and considering the $\sigma$'s and the $\tau$'s separately. \\

The first two conditions follow from Proposition \ref{wrprop}. Indeed, for a facet $H$, the group $\Gamma^+(\square^n, H) \simeq \Gamma^+(\square^{n-1})$ is generated by the rotations $\phi = (\lambda, \pi)$ by a right angle in a coordinate plane, i.e., such that $\pi = (a \ b)$ is a transposition and $\lambda$ is non-trivial only at $a$, where $H(a) = H(b) = 0$. The faces $G \preceq H$ with $G(a) = G(b) = 0$ are fixed by $\phi$, which acts trivially on $F_n / G$, while the other faces $G \preceq H$ lie in orbits of size $4$, and $\phi^4 = 1$. Therefore, all the factors in Equation \ref{monodromy} are trivial, and $\chi_i(\mu_\phi) = 1$ for all $i$. In order to address the third and fourth condition, we will study the sign of the permutation $\sigma_r \in \perm_{k_r}$ of faces of rank $r < k$ of $\square^k$ induced by an automorphism $(\lambda, \pi) \in  C_2 \wr \perm_k \simeq \Gamma(\square^k)$ for a given $k$. Suppose first that $\lambda = 1$ and $\pi = (i \ j)$ is a transposition. Then $\sigma_r$ swaps pairs of faces $F$ such that $F(i) \not = F(j)$ and fixes the other ones. There are 
\begin{equation}\label{firstcomb}
2^{k-r}\left( \frac{1}{2} \binom{k-2}{r} + 2 \binom{k-2}{r-1}  \right)
\end{equation}
such faces for $r < k-1$ while there are $4$ for $r= k-1$. Therefore, $\sigma_r$ factorizes into half as many transpositions. Suppose now that $\pi = 1$ while $\lambda_i = 1$ for all $i$ except for some $j$. Then $\sigma_r$ swaps pairs of faces $F$ such that $F(j) \not = 0$ and fixes the other ones. There are 
\begin{equation}\label{secondcomb}
2^{k-r}\binom{k-1}{r} 
\end{equation}
such faces, and thus $\sigma_r$ factorizes into half as many transpositions. Note that half of both Equation \ref{firstcomb} and \ref{secondcomb} is even for $r < k-2$. Moreover, for $r = k-2$ half of Equation \ref{firstcomb} is odd while half of Equation \ref{secondcomb} is even, and the opposite is true for $r=k-1$. All this implies that if Equation \ref{cubeautom} holds for an arbitrary automorphism $(\lambda, \pi)$ then $\sigma_r \in \alt_{k_r}$ for $r < k-2$, while ${\rm sign}(\sigma_{k-1}) = {\rm sign}(\sigma_{k-2})$. The desired third and fourth condition follow by applying the latter fact to the moves from Definition \ref{moves} with $k=n-1$.
\end{proof}
% The converse of the above result can be proven by induction on $n$, analogously to the case of the Rubik's simplex (Proposition \ref{simpmain}). The case of the Rubik's hypercube is extremely similar, with a few additional details. Even further, the embedding described in Section \ref{seccomparison} can be exploited for recycling the results from Section \ref{secfundsimp} and extend them to the hypercube. For the sake of simplicity, we briefly outline the main additional subtleties that arise. For a proof complete of all the details, we refer the reader to \cite{vladislav2020generalization}. To begin with, Proposition \ref{cubemain} is well-known for the Rubik's cube (i.e., $n=3$), which is analogous to Lemma \ref{basesimp}. The analogue of Lemma \ref{solveperm} can be proven in the same way via Proposition \ref{alternprop}. Next, Lemma \ref{bringcolors} can be extended to the Rubik's cube via the aforementioned embedding. The only difference is that, as an additional step, it is necessary to flip the sign at a given index of the involved locations, which is straightforward. From here, the analogue of Lemma \ref{solveorient} can be derived by induction with a similar argument based on commutators, where, similarly to Lemma \ref{sovleridge}, the orientations of the elements located at ridges are dealt with by bringing two given ridges to a common location and reducing to the case $n=3$.

The converse of the above result can be proven by induction on $n$, analogously to the case of the Rubik's simplex (Proposition \ref{simpmain}). The case of the Rubik's hypercube is extremely similar, with a few additional details. Even further, the embedding described in Section \ref{seccomparison} can be exploited for recycling the results from Section \ref{secfundsimp} and extend them to the hypercube. For the sake of simplicity, we briefly outline the main additional subtleties that arise. For a complete proof, we refer the reader to \cite{vladislav2020generalization}. To begin with, Proposition \ref{cubemain} is well-known for the Rubik's cube and tesseract (i.e., $n=3,4$), which is analogous to Lemma \ref{basesimp}. We thus assume $n \geq 5$. The analogue of Lemma \ref{sovleridge} follows from the case $n = 3$ via an embedding $\rubg(\square^3) \hookrightarrow \rubg(\square^n)$, which plays the role of Equation \ref{embsimp}. Namely, for a set $S$ of three indices, the moves associated to the rotations by a right angle in the plane of $S \setminus \{ \alpha \}$ of the six facets whose non-vanishing entry has index $\alpha \in S$ act on faces only through their entries with index in $S$, and thus generate a copy of $\rubg(\square^3)$. Up to conjugation by a move, the non-vanishing entries of any two ridges have indices in a common $S$. Moreover, the analogue of Lemma \ref{solveperm} can be proven in the same way via Proposition \ref{alternprop}, whose hypothesis is satisfied since every face incident to two facets of $\square^n$ sharing a ridge is incident to the latter.
Next, Lemma \ref{bringcolors} can be extended to the Rubik's hypercube for $r \geq 1$, which is the only case needed, by applying it via $\Theta$ to the sets $V(F_i)$, after a preliminary move ensuring $V(F_1) \neq V(F_2)$. This yields an index $h$ at which both locations have non-vanishing entries, different from the indices of their colors. If these entries have opposite signs, they can be matched by a half-turn in a coordinate plane containing $h$ of a facet incident to only one of the two locations, which exists since $V(F_1) \neq V(F_2)$.
Lastly, the analogue of Lemma \ref{solveorient} follows from the same argument based on commutators, with two modifications. First, the quotient $\rubg(H / F_{-1}) / L$ is a quotient of the Abelianization of $\Gamma^+(\square^{n-2})$, which has order $2$, since the latter is generated by the rotations by a right angle in a coordinate plane, which are conjugate to each other and to their inverses. As the associated moves induce odd permutations of the faces of rank $n-3$ by the proof of Lemma \ref{cubefirstlemm}, while commutators induce even ones, $L$ consists of the elements inducing even permutations of these faces, rather than containing all the elements of order $2$. This suffices, since the argument is applied to elements with trivial $\sigma$'s. Second, for $n = 5$ and $r = 0$, the inductive hypothesis is applied to the facet $H / F_{-1} \simeq \square^4$, for which the first condition in Proposition \ref{cubemain} is not vacuous. However, it is satisfied by the required elements, whose $\tau_0$ has a single non-trivial component, given by a double transposition in $\alt_4'$.

As a consequence of Proposition \ref{cubemain}, the number of possible configurations of the Rubik's hypercube is: 
\begin{equation}
|\rubg(\square^n)| =  \frac{1}{c_n 2^{2^n + 2(n-2)}} \prod_{0 \leq i < n-1}(n-i)!^{2^{n-i}\binom{n}{i}}\left(2^{n-i}\binom{n}{i} \right)!,
\end{equation}
where:
\begin{equation}
c_n = |\alt_n / \alt_n'| = \begin{cases}
3 & n = 3,4, \\
1 & n \geq 5.
\end{cases}
\end{equation}

\section{The Rubik's Torus}\label{sectori}
Most of the polytopes considered in this work, such as polygons, hosotopes, simplices and hypercubes, are spherical, in the sense that their geometric realizations tessellate a sphere. We now consider $3$-dimensional polytopes of a different topological nature: \emph{discrete tori}, which are obtained by identifying the opposite sides of an $m \times m$ grid of squares (Figure \ref{figtorus}). The associated puzzles can be thought of as flat versions of the classical Rubik's cube wrapped around a torus. Throughout this section, $m \geq 3$, since the case $m = 2$ is degenerate.

\begin{defn}\label{deftorus}
The $m \times m$ \emph{torus} is the abstract polytope whose faces are given by
\begin{equation}
\tor_m = \mathbb{Z}_{2m}^2 \sqcup \{ F_{-1}, F_3 \},
\end{equation}
where $F_{-1}$ and $F_3$ are formal elements, which are the minimum and the maximum of $\tor_m$, respectively. We denote the $i$-th component of $F \in \tor_m \setminus \{ F_{-1}, F_3 \}$ as $F(i)$. The ordering on the remaining faces is then defined as follows: $G \preceq F$ if for $i = 1,2$, $G(i) = F(i)$ whenever $F(i)$ is even, while $G(i) \in \{ F(i) - 1, F(i), F(i) + 1 \}$ whenever $F(i)$ is odd.
\end{defn}

Geometrically, a face $F \not = F_{-1}, F_3$ corresponds to the cell with barycenter $F / 2$ of the grid of unit squares on the flat torus $\mathbb{R}^2 / m \mathbb{Z}^2$, and the ordering corresponds to the inclusion of closed cells. In particular, the rank of $F$ is the number of its odd components. Therefore, $\tor_m$ has $m^2$ vertices, $2m^2$ edges, deemed \emph{horizontal} or \emph{vertical} depending on whether their first or second component is odd, and $m^2$ facets, which are squares. The translations $F \mapsto F + t$, where $t$ has even components, are automorphisms of the torus, and act simply transitively on vertices, on facets, and on edges of each direction. The rotation $\rho(F) = (-F(2), F(1))$ around the vertex $(0,0)$ is an automorphism as well, and so is the \emph{quarter-turn} around a facet $H$, defined as
\begin{equation}\label{quarterturn}
\varphi_H(F) = \rho(F - H) + H,
\end{equation}
since $H - \rho(H)$ has even components. This rotation is usually called the \emph{$S$ transformation} of the torus in the literature. Together with the reflection $(F(1), F(2)) \mapsto (F(2), F(1))$, these automorphisms act transitively on flags, meaning that $\tor_m$ is regular. Its rotation group $\Gamma^+(\tor_m)$ consists of the $4m^2$ automorphisms of the form $F \mapsto \rho^k(F) + t$, where $k \in \mathbb{Z}_4$ and $t$ has even components. Since $\tor_m$ has $8m^2$ flags, it is orientable. We refer to $k$ as the \emph{angle} of the rotation, which defines a homomorphism $\Gamma^+(\tor_m) \rightarrow \mathbb{Z}_4$. For instance, translations have angle $0$, while $\varphi_H$ has angle $1$ and generates $\Gamma^+(\tor_m, H) \simeq \Gamma^+(H / F_{-1}) \simeq C_4$. We denote by $\mu_H = \mu_{\varphi_H}$ the move associated to the quarter-turn around a facet $H$. 

\begin{figure}[h!]
\begin{center}
\begin{tikzpicture}[scale=.75,
  sid/.style={thick, postaction={decorate}, decoration={markings, mark=at position .5 with {\arrow{Stealth[length=2.2mm]}}}},
  sidd/.style={thick, postaction={decorate}, decoration={markings, mark=at position .44 with {\arrow{Stealth[length=2.2mm]}}, mark=at position .56 with {\arrow{Stealth[length=2.2mm]}}}},
  piece/.style={black}]
% left: the torus T_5 together with a move
\fill[gray!15] (1,1) rectangle (4,4);
\fill[gray!40] (2,2) rectangle (3,3);
\draw[gray!60] (0,0) grid (5,5);
\draw[sid] (0,0) -- (5,0);
\draw[sid] (0,5) -- (5,5);
\draw[sidd] (0,0) -- (0,5);
\draw[sidd] (5,0) -- (5,5);
\foreach \a/\b/\c/\d in {2.22/2/2.78/2, 3/2.22/3/2.78, 2.22/3/2.78/3, 2/2.22/2/2.78}
  \draw[piece, line width=1pt] (\a,\b) -- (\c,\d);
\foreach \x/\y in {2/2, 3/2, 3/3, 2/3}
  \fill[piece] (\x,\y) circle (2pt);
\draw[-{Stealth[length=1.6mm]}, thick] (2.5,2.5) ++(-40:.33) arc (-40:230:.33);
\node at (2.5,2.5) {\scriptsize $H$};
% right: checkerboard coloring of T_4
\begin{scope}[shift={(8,.5)}]
\foreach \x/\y in {0/0, 2/0, 1/1, 3/1, 0/2, 2/2, 1/3, 3/3}
  \fill[gray!40] (\x,\y) rectangle ++(1,1);
\draw[gray!60] (0,0) grid (4,4);
\draw[sid] (0,0) -- (4,0);
\draw[sid] (0,4) -- (4,4);
\draw[sidd] (0,0) -- (0,4);
\draw[sidd] (4,0) -- (4,4);
\end{scope}
\end{tikzpicture}
\vspace{.3cm}
\caption{The torus $\tor_5$, together with the quarter turn around $H$ (left), and the checkerboard coloring of $\tor_4$ (right).}\label{figtorus}
\end{center}
\end{figure}

% Since $\mu_{\phi^k} = \mu_\phi^k$, these moves generate $\rubg(\tor_m)$. In terms of the physical puzzle, the Rubik's torus consists of $m^2$ corner pieces with four stickers each, $2m^2$ edge pieces with two stickers each, and $m^2$ center pieces with one sticker each, where the colors of the stickers are the facets. The move $\mu_H$ rotates by a quarter-turn the nine pieces located at the faces of $H$ (Figure \ref{figtorus}, left).

We now describe $\rubg(\tor_m)$ via the wreath product representation from Section \ref{wrrep}. For a vertex $F$, we identify $F_3 / F$ with the $4$-gon $\mathcal{A}_0 = F_3 / (0,0)$ via the translation by $-F$, and fix arbitrary isomorphisms between $F_3 / F$ and the segment $\mathcal{A}_1 = F_3 / (1,0)$. Equation \ref{wrprod} provides an embedding of $\rubg(\tor_m)$ into
\begin{equation}\label{wrtorus}
\left( C_4 \wr \perm_{m^2} \right) \times \left( C_2 \wr \perm_{2m^2} \right).
\end{equation}
We denote a typical element of Equation \ref{wrtorus} by $(\tau_i, \sigma_i)_{i = 0,1}$, the components of $\tau_i$ by $\tau_i^j$, and by $r \in \Gamma^+(\mathcal{A}_0)$ the quarter-turn of $\mathcal{A}_0$, i.e., the restriction of $\rho$.

These coordinates measure orientations through angles. Indeed, along a sequence of moves, the colors of a piece are transformed by the composition $\psi \in \Gamma^+(\tor_m)$ of the quarter-turns $\varphi_H^{\pm 1}$ of the moves $\mu_H^{\pm 1}$ involving the piece, which has angle $k$ if the piece undergoes $k$ quarter-turns, counted with sign. For a corner piece, the corresponding component of $\tau_0$ is the composition of $\psi$ with translations, which has angle $k$ and fixes $(0,0)$, i.e., it is equal to $r^k$. For an edge piece returning to its location $F$, instead, $\psi \in \Gamma^+(\tor_m, F)$ is either the identity or the half-turn around $F$, which exchanges the two facets incident to $F$. Therefore, $k$ is even, and the piece is flipped if, and only if, $k \equiv 2$ modulo $4$.

Suppose now that $m$ is even. Then $\tor_m$ admits a \emph{checkerboard coloring}, given by $\kappa(F) = (F(1) + F(2))/2$ modulo $2$ for vertices and facets $F$, which is well-defined since $4$ divides $2m$ (Figure \ref{figtorus}, right). Facets sharing an edge have different colors, and the quarter-turns $\varphi_H$ preserve the colors of facets. Therefore, every sticker always lies on a facet with the same color as the facet it originates from, i.e., for every $\mu \in \rubg(\tor_m)$ and every $(F, G) \in \rub(\tor_m)$ with $G$ a facet:
\begin{equation}\label{checker}
\mu(F, G) = (F', G') \ \Longrightarrow \ \kappa(G') = \kappa(G).
\end{equation}
Thus, the orientation of an edge piece is determined by its location, and the one of a corner piece is determined up to a half-turn. In other words, the components of $\tau_1$ are determined by $\sigma_1$, and the ones of $\tau_0$ are determined by $\sigma_0$ up to multiplication by $r^2$. The main result of this section is the following complete description of the group of the Rubik's torus.

\begin{prop}\label{torusmain}
The Rubik's group of the torus $\rubg(\tor_m)$ consists exactly of the elements $(\tau_i, \sigma_i)_{i = 0,1}$ satisfying:
\begin{itemize}
\item $\prod_{1 \leq j \leq m^2} \tau_0^j = 1$.
\item $\prod_{1 \leq j \leq 2m^2} \tau_1^j = 1$.
\item ${\rm sign}(\sigma_0) = {\rm sign}(\sigma_1)$.
\item Equation \ref{checker} holds, if $m$ is even.
\end{itemize}
\end{prop}

The first three conditions are analogous to the classical constraints for the Rubik's cube, i.e., to Proposition \ref{cubemain} in dimension $3$, except that corner pieces admit four orientations instead of three. The fourth one has no analogue for the cube, and reflects the topology of the torus, in the spirit of the analogy with holonomy groups discussed after Example \ref{rubdihedra}. Even though the torus is flat, a corner piece can return to its location rotated by a quarter-turn, and an edge piece can return flipped, if, and only if, $m$ is odd. As we will see, this is achieved by transporting the pieces around a closed row of $m$ facets, in which consecutive facets share an edge. For odd $m$, such a row also rules out a checkerboard coloring.

\begin{lemm}\label{torusfirstlemm}
The elements in $\rubg(\tor_m)$ satisfy the conditions in Proposition \ref{torusmain}.
\end{lemm}

\begin{proof}
Since the conditions define a subgroup of Equation \ref{wrtorus}, and the fourth one has been discussed above, it is sufficient to check the first three for the moves $\mu_H$. The four vertices of $H$, as well as its four edges, form a single orbit under $\varphi_H$, on which $\mu_H$ induces a $4$-cycle. Since $\varphi_H^4 = 1$, all the factors in Equation \ref{monodromy} are trivial, which implies the first two conditions by Proposition \ref{wrprop}. Since $4$-cycles are odd permutations, the third condition holds as well.
\end{proof}

Next, we describe the permutations of locations, by relying on the results of Section \ref{descent}. Moreover, we show that pieces of the same kind can be cycled in triples, without affecting the rest of the puzzle.

\begin{lemm}\label{toruspositions}
The image of the projection $\rubg(\tor_m) \rightarrow \perm_{m^2} \times \perm_{2m^2}$ consists of the pairs $(\sigma_0, \sigma_1)$ such that ${\rm sign}(\sigma_0) = {\rm sign}(\sigma_1)$. Moreover, for any three distinct faces $F, F', F''$ of $\tor_m$ of the same rank $i \in \{0, 1\}$, there exists an element of $\rubg(\tor_m)$ that fixes all the elements of $\rub(\tor_m)$ located at faces different from $F, F', F''$, and such that $\sigma_i$ is the $3$-cycle $(F \ F' \ F'')$.
\end{lemm}

\begin{proof}
We apply Proposition \ref{alternprop}. Since $m \geq 3$, every face incident to two facets of $\tor_m$ sharing an edge is incident to that edge, i.e., the hypothesis of Proposition \ref{extendprop} is satisfied. Moreover, the facets $H$ and $H + (2, 2)$ have exactly one incident vertex in common, namely $H + (1, 1)$, and no common incident edge. Therefore, condition (M) holds for vertices, and thus for edges as well, which implies that the image contains $\alt_{m^2} \times \alt_{2m^2}$. Together with the pair of $4$-cycles induced by any move, these generate the subgroup of pairs of permutations with the same sign. The other inclusion follows from Lemma \ref{torusfirstlemm}.

For the second claim, recall from the proof of Proposition \ref{alternprop} that if two permutations $\alpha, \beta$ have supports intersecting in a single point $x$, then their commutator $[\alpha, \beta] = \alpha^{-1} \beta^{-1} \alpha \beta$ is the $3$-cycle $(x \ \ \alpha^{-1}(x) \ \ \beta^{-1}(x))$. The same holds for permutations of blocks of elements -- here, of the elements of $\rub(\tor_m)$ with a given location -- in which case the commutator fixes all the elements outside the three blocks involved. Therefore, for $i = 0$, the commutator $[\mu_H, \mu_{H + (2,2)}]$ induces a $3$-cycle on vertices, and fixes all the elements located at the other faces. For $i = 1$, consider instead the facets $H$ and $K = H + (2, 0)$, whose only common edge is $H + (1, 0)$. Similarly, $c = [\mu_H, \mu_K]$ induces a $3$-cycle on edges, and fixes all the elements located at the other edges. Moreover, a direct computation shows that it induces a product of two disjoint transpositions on vertices. Thus, $c^2$ fixes all the vertices, and acts on the elements located at each of them via a rotation of order dividing $4$, while $c^6$ fixes all the elements located at edges. Therefore, $c^8$ fixes all the elements located at vertices, and acts on the ones located at edges as $c^2$, which induces a $3$-cycle. Lastly, by the first claim, the permutations of the faces of rank $i$ induced by $\rubg(\tor_m)$ contain all the even ones, which act $3$-transitively since there are at least $m^2 \geq 9$ such faces. Therefore, by conjugating the above elements via appropriate elements of $\rubg(\tor_m)$, we obtain the desired elements for any $F, F', F''$.
\end{proof}

Lastly, we deal with orientations, by exploiting their interpretation as angles.

\begin{lemm}\label{torusorient}
For any $\tau_0 \in C_4^{m^2}$ and $\tau_1 \in C_2^{2m^2}$ such that the element $(\tau_i, \sigma_i = 1)_{i = 0,1}$ satisfies the conditions in Proposition \ref{torusmain}, the latter belongs to (the image of the embedding of) $\rubg(\tor_m)$.
\end{lemm}

\begin{proof}
We rely on commutators once again. Let $F, F', F''$ be three distinct faces of rank $i \in \{0, 1\}$, and let $\mu$ be as in Lemma \ref{toruspositions}. Suppose that $g \in \rubg(\tor_m)$ fixes the locations $F$, $F'$ and $F''$, and denote by $a, a', a''$ the corresponding components of the $\tau_i$ of $g$. Since $C_4$ and $C_2$ are Abelian, a direct computation shows that the commutator $[g, \mu]$ induces trivial permutations of locations, fixes all the elements located at faces different from $F, F', F''$, and the components of its $\tau_i$ at $F, F', F''$ are $a^{-1} a'$, $a'^{-1} a''$ and $a''^{-1} a$, respectively. In particular, if $a' = a''$, then $[g, \mu]$ rotates the pieces located at $F$ and $F''$ by $a^{-1} a'$ and by its inverse, respectively, while fixing everything else.

First, let $H$ be a facet and $H' = H + (2, 0)$. The element $g = \mu_{H'} \mu_H$ brings the corner piece located at the common vertex $F = H + (1, -1)$ back to $F$ via two quarter-turns, and fixes all the elements located at vertices not incident to $H$ or $H'$, such as $F' = F + (0, 4)$ and $F'' = F + (2, 4)$. Therefore, $a = r^2$ and $a' = a'' = 1$, and $[g, \mu]$ rotates the corner pieces located at $F$ and $F''$ by a half-turn.

Suppose now that $m$ is odd, and let $g = \mu_{H_{m-1}} \cdots \mu_{H_1} \mu_{H_0}$, where $H_j = (2j + 1, 1)$ for $0 \leq j < m$ form a closed row of facets. A direct computation shows that $g$ carries the corner piece located at $F = (0,0)$ along the row back to $F$ via $m$ quarter-turns, while for $0 < j < m$ it moves the one located at $(2j, 0)$ to $(2j, 2)$ and back via two quarter-turns. For $F' = (2, 0)$ and $F'' = (4, 0)$, we obtain $a = r^m$ and $a' = a'' = r^2$, so that $[g, \mu]$ rotates the corner pieces located at $F$ and $F''$ by the quarter-turns $r^{2-m}$ and $r^{m-2}$, respectively. Similarly, $g$ cycles the edges $(1, 0), (3, 0), \cdots, (2m-1, 0)$ and the vertical edge $(0, 1)$, in this order, and an edge piece undergoes $2m$ quarter-turns along this cycle of length $m + 1$. Since $m$ is odd, $g^{m+1}$ fixes and flips the edge pieces located at $F' = (1, 0)$ and $F'' = (3, 0)$, while it fixes all the elements located at edges not incident to the $H_j$'s, such as $F = (1, 4)$. Therefore, $[g^{m+1}, \mu]$ flips the edge pieces located at $F$ and $F''$.

Lastly, since $C_4$ and $C_2$ are Abelian, conjugating an element with trivial $\sigma_0$ and $\sigma_1$ by an element of $\rubg(\tor_m)$ permutes the components of its $\tau_i$'s according to the $\sigma_i$'s of the latter, which can be arbitrary even permutations by Lemma \ref{toruspositions}. Therefore, $\rubg(\tor_m)$ contains the elements rotating any two corner pieces by a half-turn and, if $m$ is odd, the ones rotating any two corner pieces by $r$ and $r^{-1}$, respectively, and the ones flipping any two edge pieces, while fixing everything else. If $m$ is odd, these elements generate the group of elements $(\tau_i, 1)_i$ with $\prod_j \tau_0^j = 1$ and $\prod_j \tau_1^j = 1$, as desired. If $m$ is even, Equation \ref{checker} forces $\tau_1 = 1$ and $\tau_0^j \in \langle r^2 \rangle$ for every element of the form $(\tau_i, 1)_i$, and the half-turns of pairs of corner pieces generate the group of such elements with $\prod_j \tau_0^j = 1$.
\end{proof}

We are now ready to conclude. By Lemma \ref{torusfirstlemm}, $\rubg(\tor_m)$ is contained in the subgroup of Equation \ref{wrtorus} defined by the conditions in Proposition \ref{torusmain}. Conversely, given an element $(\tau_i, \sigma_i)_i$ of the latter, Lemma \ref{toruspositions} provides $\mu \in \rubg(\tor_m)$ with the same $\sigma_0$ and $\sigma_1$, and $(\tau_i, \sigma_i)_i \, \mu^{-1}$ belongs to $\rubg(\tor_m)$ by Lemma \ref{torusorient}. This proves Proposition \ref{torusmain}. As a consequence, the number of possible configurations of the Rubik's torus is:
\begin{equation}\label{torusorder}
|\rubg(\tor_m)| = \begin{cases}
2^{4m^2 - 4} \, (m^2)! \, (2m^2)! & {\rm if} \ m \ {\rm is \ odd}, \\
2^{m^2 - 2} \, (m^2)! \, (2m^2)! & {\rm if} \ m \ {\rm is \ even}.
\end{cases}
\end{equation}
Indeed, there are $(m^2)! \, (2m^2)! / 2$ possible pairs of permutations $(\sigma_0, \sigma_1)$ by Lemma \ref{toruspositions}, and, for each of them, there are $4^{m^2 - 1} \, 2^{2m^2 - 1}$ or $2^{m^2 - 1}$ possible orientations, respectively. 
% For instance, the $3 \times 3$ Rubik's torus has $2^{32} \cdot 9! \cdot 18! \approx 10^{31}$ configurations. As a sanity check, we have verified Equation \ref{torusorder} for all $3 \leq m \leq 6$ by computing the order of the permutation group generated by the moves via the Schreier--Sims algorithm.

When $m$ is even, the second condition in Proposition \ref{torusmain} follows from the fourth one, and the description can be simplified further by choosing coordinates compatible with the checkerboard coloring. Namely, we compose the translation with $r$ for the vertices $F$ with $\kappa(F) = 1$, and choose the isomorphisms $\mathcal{A}_1 \simeq F_3 / F$ that preserve the colors of the facets. In these coordinates, Equation \ref{checker} is equivalent to $\tau_1 = 1$ and $\tau_0^j \in \langle r^2 \rangle \simeq C_2$ for all $j$. As a result, $\rubg(\tor_m)$ is isomorphic to the subgroup of $\left( C_2 \wr \perm_{m^2} \right) \times \perm_{2m^2}$ consisting of the elements $(\xi, \sigma_0, \sigma_1)$ such that $\chi(\xi, \sigma_0) = 1$ and ${\rm sign}(\sigma_0) = {\rm sign}(\sigma_1)$, where $\chi \colon C_2 \wr \perm_{m^2} \rightarrow C_2$ denotes the canonical character.

\section*{Acknowledgements}
This work was partially supported by the Wallenberg AI, Autonomous Systems and Software Program (WASP) funded by the Knut and Alice Wallenberg Foundation. 

\newpage 

{
% \small
\bibliographystyle{plainnat}
\bibliography{main}
}
% \newpage
% \appendix

%%%%%%%%%%%%%%%%%%%%%%%%%%%%%%%%%%%%%%%%%%%%%%%%%%%%%%%%%%%%

\end{document}